\documentclass[11pt,a4paper]{article}
\usepackage{amsmath,amssymb,amsthm}

\usepackage{graphicx}
\graphicspath{{./figures/}}
\usepackage{bm}
\usepackage{algorithm,algorithmic}
\usepackage[utf8]{inputenc} % allow utf-8 input
\usepackage[T1]{fontenc}    % use 8-bit T1 fonts
\usepackage{hyperref}       % hyperlinks
\usepackage{url}            % simple URL typesetting
\usepackage{booktabs}       % professional-quality tables
\usepackage{amsfonts}       % blackboard math symbols
\usepackage{microtype}      % microtypography
\usepackage{enumerate}
\usepackage{color}
\usepackage[numbers]{natbib}

\def\of{\phi}
\def\oL{L_{\phi}}
\newcommand{\bvec}[1]{\mbox{\boldmath $#1$}}
\newcommand{\prox}{{\mbox{\boldmath $\rm prox$}}}
\newcommand{\proj}{{\mbox{\boldmath $\rm proj$}}}
\newcommand{\sign}{{\rm sign}}

\newcommand{\dom}{\mathrm{dom}}

\newcommand{\argmin}{\operatornamewithlimits{argmin}}
\newcommand{\knorm}[1]{|\!|\!|#1|\!|\!|}
\newtheorem{thm}{Theorem}
\newtheorem{proposition}{Proposition}
\newtheorem{examp}{Example}
\newtheorem{lemma}{Lemma}
\newtheorem{assumption}{Assumption}
\newtheorem{rem}{Remark}

\usepackage{ifthen}
\newcommand{\COMM}[2]{{
\ifthenelse{\equal{#1}{KT}}{\color{blue}}{
\ifthenelse{\equal{#1}{AT}}{\color{red}}{
\ifthenelse{\equal{#1}{JG}}{\color{magenta}}}}
[#1: #2]
}}

\begin{document}
%\METRmakecover
%----------------------------------------------------------------------

\title{Efficient DC Algorithm for Constrained Sparse Optimization}
\author{Katsuya TONO\\
	\\
        Department of Mathematical Informatics,\\
        The University of Tokyo\\
        \texttt{\normalsize katsuya\_tono@mist.i.u-tokyo.ac.jp}\\
        \\
        Akiko TAKEDA\\
        \\
        Department of Mathematical Analysis and Statistical Inference,\\
              The Institute of Statistical Mathematics\\
        \texttt{\normalsize atakeda@ism.ac.jp}\\
        \\
        Jun-ya GOTOH\\
	\\
	Department of Industrial and Systems Engineering,\\
	Chuo University\\
	\texttt{\normalsize jgoto@indsys.chuo-u.ac.jp}}
	\date{January 30, 2017}
\maketitle

\begin{abstract}
We address the minimization of a smooth objective function under an $\ell_0$-constraint and simple convex constraints.
When the problem has no constraints except %for 
the $\ell_0$-constraint,
%there are 
some 
efficient algorithms 
are available; for example, Proximal DC (Difference of Convex functions) Algorithm (PDCA) %are proposed
repeatedly %computes 
evaluates closed-form solutions of convex subproblems, 
%which leads 
leading to a stationary point of the $\ell_0$-constrained problem. 
However, when the problem has additional convex constraints, they become inefficient
because it is difficult to obtain closed-form solutions of the %resulting 
associated subproblems. 
In this paper, we reformulate the problem %using 
by employing 
a new DC representation of the $\ell_0$-constraint, 
%by the difference of convex (DC) functions, which 
so that %the 
PDCA can %retains 
retain the efficiency %of PDCA 
by reducing its subproblems to the projection operation onto a convex set. 
Moreover, inspired by the Nesterov's acceleration technique %in the context of 
for proximal methods, we propose the Accelerated PDCA (APDCA), which %has 
attains the optimal convergence rate if applied to convex programs, and performs well in numerical
experiments. %terms.
\end{abstract}

\section{Introduction}\label{introduction}
\subsection{Background}

In recent years, sparse optimization problems which include the $\ell_{0}$-norm of decision vector in their objectives or constraints have drawn significant attentions in many applications
% including 
 such as signal processing, bioinformatics, and machine learning.
Since such problems are intractable due to the nonconvexity and discontinuity of the $\ell_{0}$-norm~\cite{natarajan1995sparse}, many approaches  have been proposed to approximate the $\ell_{0}$-norm.
The $\ell_{1}$-norm regularization, initiated by Tibshirani~\cite{tibshirani1996regression} for linear regression, has been at the center of sparse optimization.
However, the $\ell_{1}$-regularizer %cannot 
does not always capture the true relevant variables since it can be a loose relaxation of the $\ell_{0}$-norm~\cite{candes2008enhancing}.
To overcome this drawback, many regularizers which abandon the convexity have been proposed to %surrogate 
approximate the $\ell_{0}$-norm in better ways.
Typical examples are Smoothly Clipped Absolute Derivation (SCAD)~\cite{fan2001variable}, Log-Sum Penalty (LSP)~\cite{candes2008enhancing}, Minimax Concave Penalty (MCP)~\cite{zhang2010nearly}, and capped-$\ell_{1}$ penalty~\cite{zhang2010analysis}.
%These regularizers are separable (or elementwise) with respect to variables.
%Meanwhile, inseparable regularizers have also been proposed.
%The group $\ell_{1}$ penalty~\cite{yuan2006model} promotes structured sparsity, i.e., it promotes sparsity patterns based on a known grouping structure of the variables. 
%The group SCAD and the group MCP~\cite{wang2008variable} are their nonconvex counterparts. 

On the other hand, there are some approaches which do not approximate the $\ell_0$-norm;
DC (Difference of Convex functions) optimization approaches, employed in~\cite{thiao2010dc, gotoh2015dc},
 replace the $\ell_0$-norm by a difference of two convex functions and then apply the DC algorithm (DCA)~\cite{pham1997convex} (also known as Convex-ConCave Procedure (CCCP)~\cite{yuille2003concave} or the Multi-Stage (MS) convex relaxation~\cite{zhang2010analysis}) to the resulting DC program.
However, as some papers including \cite{gong2013general,luo2015new} pointed out, DCA requires solving a sequence of convex subproblems, often resulting in a large computation time.
%Recently, the difference of two norms is used as a nonconvex sparse regularizer: 
%the difference of $\ell_1$ and $\ell_2$ norms, known as
%$\ell_{1-2}$ penalty~\cite{yin2015minimization}, and the difference of $\ell_{1}$ and largest-$K$ norms
%($\ell_{1-[K]}$ penalty)~\cite{luo2015new,gotoh2015dc}, which promotes sparse vectors with at most $K$ nonzero elements for a given integer $K$.
%The DC (Difference of Convex functions) algorithm~\cite{pham1997convex} (or Convex-ConCave Procedure (CCCP)~\cite{yuille2003concave}, or the Multi-Stage (MS) convex relaxation~\cite{zhang2010analysis}) is widely applied to such nonconvex and nonsmooth optimization problems.
%However, as some existing works including \cite{gong2013general,luo2015new} point out, it requires solving a sequence of convex subproblems, often resulting in a large computation time.

When the problem has no additional constraints other than the $\ell_0$-norm constraint, this issue can be resolved; %there proposed some iterative algorithms which compute closed-form solutions of simpler subproblems.
%some algorithms have been proposed, where closed-form solutions of simpler subproblems are applied.
some algorithms whose subproblems have closed-form solutions have been proposed.
%Our previous work~\cite{gotoh2015dc} 
Gotoh et al.~\cite{gotoh2015dc}
transformed the problem %with no 
without convex constraints into an equivalent problem minimizing a DC objective function;
%and %proposed to employ %the 
a special DC decomposition %for the transformed DC objective function
is employed 
so that its subproblems can be reduced to the so-called soft-thresholding operations, which can be carried out in linear time. 
The resulting DCA is called the Proximal DC Algorithm (PDCA), which constitutes a special case of the framework of Sequential Convex Programming (SCP)~\cite{lu2012sequential}.
Iterative Hard-Thresholding (IHT) algorithm~\cite{bertsimas2016best} is another efficient method for the $\ell_0$-constrained optimization. 
In IHT algorithm, we repeat solving subproblems of minimizing a quadratic surrogate function under the $\ell_0$-norm constraint, whose solutions are simply obtained by the so-called hard-thresholding operation.

Some applications of sparse optimization have convex constraints 
%However, if there are some convex constraints
such as the $\ell_2$-norm constraint and nonnegative constraint other than the $\ell_0$-norm constraint.
For such constrained sparse optimization,
all the algorithms mentioned above %have such difficult subproblems that they have no closed-form solutions in general.
generate a sequence of convex subproblems whose closed-form solutions cannot be readily available in general. %obtained in general.
To overcome this issue, we propose a new DC representation of the $\ell_0$-constraint,
%a novel DC representation of $\ell_0$-norm constraint
which leads to a PDCA whose subproblems have closed-form solutions.
%so that PDCA for the transformed problem has simpler subproblems with closed-form solutions.

%On the other hand, if the problem is a large-scale convex optimization, first order methods such as
%the Proximal Gradient Method (PGM) are very popular.
%From the computational point of view, the Accelerated Proximal Gradient (APG)~\cite{beck2009fast} method, the accelerated variant of PGM, has been central.
%Recently, PGM has been extended to deal with nonconvex programs, which is named the General Iterative Shrinkage and Thresholding (GIST) algorithm~\cite{gong2013general}.
%Inspired by GIST, APG has also been generalized to nonconvex optimization problems~\cite{li2015accelerated}, which we call AGIST in this paper.
%AGIST is the first APG-type algorithm for nonconvex optimization problems with guarantees that it converges to a stationary point and its convergence rate retains $\mathrm{O}(1/k^2)$ if it is applied to convex problems.
%However, both GIST and AGIST involve computing a proximal operator of a nonconvex regularizer, which makes GIST applicable only to a limited class of simply regularized problems.

\subsection{Contributions}
We propose an efficient approach to the constrained sparse optimization: the minimization of an objective function under the $\ell_0$-constraint and some convex constraints.
%Our idea is based on the DC decomposability of $\ell_0$-norm constraint.
Gotoh et al.~\cite{gotoh2015dc}
%Our previous work~\cite{gotoh2015dc}
proposed to express the $\ell_0$-norm as a difference of two convex functions as $\of_1-\of_2$, both of which are nonsmooth.
However, in applying PDCA to such a constrained problem, the nonsmoothness of the first term $\of_1$ collides with the convex constraints, resulting in making the subproblems difficult to have closed-form solutions.
In this paper, we rewrite the $\ell_0$-norm constraint as % a difference of two convex functions
another DC function so that the former convex function $\of_1$ is smooth.
In applying PDCA, the smoothness of the former term makes subproblems easily %solved 
solvable by a projection operation
onto the convex set.

To achieve faster convergence, we further propose the Accelerated version of PDCA (APDCA), inspired by the preceding work~\cite{li2015accelerated} on extending the Accelerated Proximal Gradient (APG) method (originally for convex program) to nonconvex program.
We construct APDCA by employing techniques used in the nonmonotone APG~\cite{li2015accelerated} for nonconvex program,
so the convergence results for the nonmonotone APG can be shown to hold for APDCA; (i) APDCA has the convergence rate of $\mathrm{O}(1/t^2)$, if applied to convex program, where $t$ denotes the iteration counter, (ii) APDCA has the subsequential convergence to a stationary point.
%APDCA has the subsequential convergence to a stationary point.

In the numerical section, we demonstrate the numerical performance of our approach compared to the existing DC optimization approaches.
The efficiency of APDCA applied to our reformulation is confirmed %on 
with three typical examples of the constrained sparse optimization, using both synthetic and real-world data.

%This implies that PDCA has a wider range of applications than GIST; PDCA is applicable as far as the solution to the proximal operator problem for the convex term and the subgradient of the concave term are available.
%In addition, we also propose two types of the accelerated version of PDCA (APDCA).
%See Table~\ref{table:relatedwork} for the relation of
%our methods to existing ones.
%To sum up, we propose the following three algorithms:
%\begin{itemize}
%\item PDCA is a generalized PGM including two preceding algorithms~\cite{gotoh2015dc,luo2015new}
%that are proposed for $\ell_{1-[K]}$ penalty as special cases.
%Especially if we fix its step size, PDCA is none other than a kind of DC algorithm.
%\item monotone APDCA is an accelerated variant of PDCA with decreasing objective values. It has $\mathrm{O}(1/t^2)$ convergence rate under convexity.
%\item nonmonotone APDCA is another variant of accelerated PDCA with less computation per each iteration, while an evaluation of convergence is a bit complicated.
%\end{itemize}

The remainder of this paper is %constructed 
structured as follows.
In Section~\ref{sec:preliminaries}, we define the constrained sparse optimization problem and review some existing approaches.
In Section~\ref{sec:representation}, we propose a DC %expression
representation of the $\ell_0$-norm constraint and then show how to apply PDCA to the transformed problem.
In Section~\ref{sec:link}, we show %the 
a close relation between PDCA and Proximal Gradient Method (PGM), which %extends 
motivates us to extend the framework of PDCA. 
In Section~\ref{sec:acceleration}, we accelerate PDCA to achieve faster convergence and review some related algorithms.
In Section~\ref{sec:experiment}, we demonstrate the efficiency of our methods %compared to 
in comparison with other DCA frameworks.

%\COMM{KT}{Notationをここに加筆？}

\section{Preliminaries}\label{sec:preliminaries}
\subsection{Problem settings}
In this paper, we address the following $\ell_0$-constrained problem:
\begin{align}
%\min_{\bvec x\in C}\left\{\of(\bvec x):\|\bvec x \|_0\le k\right\},\label{eq:l0const}
\min_{\bm{x}}\big\{\,\of(\bm{x})\,:\,\|\bm{x}\|_0\leq k,~\bm{x}\in C\,\big\},\label{eq:l0const}
\end{align}
where $\of:\mathbb R^{n}\to\mathbb R$, $k\in \{1,\ldots, n\}$, $\|%\cdot
\bm{x}\|_0$ denotes the number of nonzero elements (called the $\ell_0$-norm or the cardinality) of a vector $\bm{x}$, and $C\subseteq \mathbb R^n$ is a nonempty closed convex set. 
A solution $\bm{x}$ satisfying the $\ell_0$-constraint, $\|\bm{x}\|_0\leq k$, is said to be {\it $k$-sparse}. 

Throughout the paper, we make the following assumptions.

\begin{assumption}\label{assump}
\begin{enumerate}[{(}a{)}]
\item $\of%(\bvec x)
$ is continuously differentiable with $\oL$-Lipschitz continuous gradient, i.e., there exists a constant $\oL$ such that
\begin{align*}
\|\nabla \of(\bvec x)-\nabla \of(\bvec y)\|_2\le \oL\|\bvec x-\bvec y\|_2\quad (\bvec x,\bvec y\in\mathbb R^{n}),
\end{align*}
where $\|\bm{x}\|_2$ denotes the $\ell_2$-norm of $\bm{x}$.

\item  The projection $\proj_C(\bvec u)$ of a point $\bvec u\in\mathbb R^n$ onto $C$ %is easily available
can be evaluated efficiently:
\begin{align*}
\proj_C(\bvec u):=\argmin_{\bvec x\in C}\left\{ \frac{1}{2}\|\bvec x-\bvec u\|_2^2 \right\}.
\end{align*}

\item $\phi(\bvec x)+I_C(\bvec x)$ is bounded from below and coercive, i.e., $\phi(\bvec x)+I_C(\bvec x)\to \infty$ as $\|\bvec x\|_2\to \infty$, where $I_C$ denotes the indicator function of $C$ defined as
\begin{align*}
I_C(\bvec x):=\begin{cases}
0,&(\bvec x\in C),\\
+\infty,&(\bvec x\notin C).
\end{cases}
\end{align*}%.

\item The feasible region $\{\bvec x\in C:\|\bvec x\|_0\le k\}$ of~\eqref{eq:l0const} is nonempty.

%\item  $F(\bvec x)$ is coercive, i.e., $F$ is bounded from below and
%\begin{align*}
%$F(\bvec x)\to \infty \quad(\|\bvec x\|_2\to \infty).$
% \end{align*}
\end{enumerate}
\end{assumption}
Various problems in many application areas are formulated as~\eqref{eq:l0const}.
\begin{examp}[sparse principal component analysis~\cite{thiao2010dc}]\label{ex:sparsePCA}
Let $\bvec V \in \mathbb R^{n\times n}$ be a covariance matrix.
When
\begin{align*}
\of(\bvec x)=-\bvec x^\top V\bvec x,\ C=\{\bvec x\in \mathbb R^n: \|\bvec x\|_2\le 1\},
\end{align*}
%the 
problem~\eqref{eq:l0const} is called the sparse Principal Component Analysis (PCA).
In sparse PCA, we seek a $k$-sparse vector that approximates the %largest 
eigenvector which corresponds to the largest eigenvalue and regard it as the first principal component.

\end{examp}
\begin{examp}[sparse portfolio selection]\label{ex:sparseportfolio} %\COMM{KT}{要文言変更．}
Let $\bvec V\in \mathbb R^{n\times n}$ be a covariance matrix, $\bvec r\in \mathbb R^n$ a mean %return 
vector of returns of investable assets, and $\alpha>0$ a risk-aversion parameter.
When
\begin{align*}
\of(\bvec x)=\alpha\bvec x^\top\bvec V\bvec x-\bvec r^\top\bvec x,\ C=\{\bvec x\in \mathbb R^n: \bvec 1^\top \bvec x=1\},
\end{align*}
where $\bvec 1\in\mathbb R^n$ denotes the all-one vector, %the 
problem~\eqref{eq:l0const} can be seen as a variant of the sparse portfolio selection (e.g., \cite{gulpinar2010robust,takeda2013simultaneous}).
\end{examp}

\begin{examp}[sparse nonnegative linear regression]\label{ex:sparsennreg}
Let $\bvec A\in \mathbb R^{m\times n}$, $\bvec b\in\mathbb R^m$, and $I\subseteq\{1,\ldots,n\}$.
When
\begin{align*}
\of(\bvec x)=\frac{1}{2}\|\bvec A\bvec x-\bvec b\|_2^2,\ C=\{\bvec x\in \mathbb R^n: x_i\ge 0\quad (i\in I)\},
\end{align*}
%the 
problem~\eqref{eq:l0const} is the sparse nonnegative linear regression problem.
This %model 
problem includes the following problems as special cases: the ordinary least squares problem with variable selection ($I=\emptyset$) %~\cite{bertsimas2016best}
and the sparse least squares problem with all variables nonnegative ($I=\{1,\ldots,n\}$)~\cite{slawski2013non}.
\end{examp}

\subsection{Existing approaches to $\ell_0$-constrained optimization}
%\subsubsection{DC optimization approaches}
\subsubsection{%For 
Case for general $\ell_0$-constrained optimization}
%Our previous work
Gotoh et al.~\cite{gotoh2015dc} proposed to express the $\ell_0$-norm constraint as a DC function: % in another way:
\begin{align*}
\|\bvec x\|_0\le k\iff \|\bvec x\|_1-\knorm{\bvec x}_{k,1}=0,  
\end{align*}
where $\knorm{\bvec x}_{k,1}$, which we call top-$(k,1)$ norm, denotes the $\ell_1$-norm of a subvector composed of top-$k$ elements in absolute value. 
Precisely, 
%Namely, 
\begin{align}
\knorm{\bvec x}_{k,1}:=|x_{\pi(1)}|+\cdots+|x_{\pi(k)}|,  \label{largeKL1norm}
%\knorm{\bvec x}_{k,1}:=|x_{(1)}|+\cdots+|x_{(k)}|,  \label{largeKL1norm}
\end{align}
where $\pi$ is an arbitrary permutation of $\{1,\ldots, n\}$ %that satisfies
such that $|x_{\pi(1)}|\ge \cdots \ge |x_{\pi(n)}|$. Namely, $x_{\pi(i)}$ denotes the $i$-th largest element of $\bvec x$ in absolute value.
%where $x_{(i)}$ denotes the $i$-th largest element of $\bvec x$.

Then \cite{gotoh2015dc} considered the following penalized problem associated with~\eqref{eq:l0const}:
\begin{align}
\min_{\bm{x}\in C}\left\{\of(\bvec x)+\rho(\|\bvec x\|_1-\knorm{\bvec x}_{k,1})\right\}.\label{eq:l1penalty}
\end{align}
and gave an exact penalty parameter under which %the 
problems \eqref{eq:l0const} and \eqref{eq:l1penalty} are equivalent
for some examples, e.g.,  $C$ is $\mathbb R^n$ and $\{\bvec x\in\mathbb R^n:\|\bvec x\|_2\le 1\}$.
%For the case where $C$ is $\mathbb R^n, \{\bvec x\in\mathbb R^n:\|\bvec x\|_2\le 1\}$, or a standard simplex $\{\bvec x\in\mathbb R^n:\bvec x\ge 0, \bvec 1^\top\bvec x=1\}$, they gave an exact penalty parameter under which %the 
Problems \eqref{eq:l0const} and \eqref{eq:l1penalty} are equivalent.
%The reformulation~\eqref{eq:l1penalty} is solved by the so-called DC algorithm (DCA).
Then the so-called DC Algorithm (DCA) is applied to the reformulation~\eqref{eq:l1penalty}. 
In general, 
to minimize a DC function $\of_1(\bvec x)-\of_2(\bvec x)$, expressed by two convex functions $\of_1$ and $\of_2$, 
DCA %for minimizing a DC function $\of_1(\bvec x)-\of_2(\bvec x)$ 
solves the following subproblem repeatedly:
\begin{align}
\bvec x^{(t+1)}\in \argmin_{%\bvec x
\bm{x}\in \mathbb R^n}\left\{ \of_1(\bvec x)-\bvec x^\top \bvec s(\bvec x^{(t)}) \right\},\label{eq:DCAsubprob}
\end{align}
where $\bvec s(\bvec x^{(t)})$ is a subgradient of $\of_2(\bvec x^{(t)})$ at $\bvec x^{(t)}$, i.e.,
\begin{align*}
\bvec s(\bvec x^{(t)})\in\partial \of_2(\bvec x^{(t)}):=\{\bvec y\in\mathbb R^n:\of_2(\bvec x)\ge \of_2(\bvec x^{(t)})+\langle \bvec y,\bvec x-\bvec x^{(t)} \rangle \quad(\bvec x\in\mathbb R^n) \}.
\end{align*}
When applying DCA to~\eqref{eq:l1penalty}, %two ways of DC decomposition of the objective function are
\cite{gotoh2015dc} used the following decomposition for a DC function $\of=\gamma-\iota$:
%and a general convex set $C$:
\begin{align}
  &\of_1(\bvec x)=\gamma(\bvec x)+\rho\|\bvec x\|_1+I_C(\bvec x),\quad \of_2(\bvec x)=\iota(\bvec x)+\rho\knorm{\bvec x}_{k,1}.
  %\quad(\text{if $f$ is convex})
  \label{eq:naiveDCdecomp}
%&f_1(\bvec x)=\left(\frac{\oL}{2}\|\bvec x\|_2^2+\rho\|\bvec x\|_1\right),\quad f_2(\bvec x)=\left(\frac{\oL}{2}\|\bvec x\|_2^2-f(\bvec x)+\rho\knorm{\bvec x}_{k,1}\right)\quad (\text{if $C=\mathbb R^n$}).\label{eq:l1proxDCdecomp}
\end{align}
%If $\of$ is not convex but DC, i.e., $\of=\gamma-\iota$, with two convex functions $\gamma,\iota$, we add the former convex term, $\gamma$, into $\of_1$ and the latter, $\iota$, into $\of_2$. 

The resulting subproblem~\eqref{eq:DCAsubprob} is a convex problem, but
because it generally does not have a closed-form solution for \eqref{eq:DCAsubprob},
we need to repeatedly apply some convex optimization algorithm to solve the convex problem, which is 
often time-consuming.

%\paragraph{}
Thiao et al. \cite{thiao2010dc} gave another DC formulation, which is based on Mixed Integer Programming (MIP).
They first rewrote the $\ell_0$-norm using a binary vector $\bvec u$ as
\begin{align*}
\|\bvec x\|_0\le k\iff |x_i|\le Mu_i\ (i=1,\ldots,n),\ \bvec 1^\top\bvec u\le k,\ \bvec u\in \{0,1\}^n,
\end{align*}
where $M$ is a so-called big-$M$ constant, which is set to be sufficiently large.
Then using the following equivalence:
\begin{align*}
\bvec u\in \{0,1\}^n\iff \bvec u\in [0,1]^n, (\bvec 1-\bvec u)^\top \bvec u\le 0,
\end{align*}
they finally obtained a penalized DC formulation of~\eqref{eq:l0const}:
\begin{align}
\min_{\bm{x}\in C}\left\{\of(\bvec x)+\rho(\bvec 1-\bvec u)^\top \bvec u:|x_i|\le Mu_i\ (i=1,\ldots,n),\ \bvec 1^\top\bvec u\le k,\ \bvec u\in [0,1]^n\right\},\label{eq:concavepenalty}
\end{align}
which is solved by DCA~\eqref{eq:DCAsubprob}.
While this approach %is 
was 
originally 
proposed just 
for Example~\ref{ex:sparsePCA}, it works also in our general settings.
%(and thus an exact penalty is computed only for the example), it works in our general settings.
We need to use some convex optimization algorithms for the resulting convex subproblem as well as
the above-mentioned DCA of
\cite{gotoh2015dc}. 
%for repeatedly solving subproblems.

%\subsubsection{Iterative hard-thresholding}
\subsubsection{Case for %For 
$\ell_0$-constrained optimization without other constraints}
 For the case where $C=\mathbb R^n$, 
%the 
paper \cite{gotoh2015dc} proposed a different %type of 
DC decomposition, $\of_1-\of_2$, where %: for the case where $C=\mathbb R^n$:
\begin{align}
  &\of_1(\bvec x)=\left(\frac{\oL}{2}\|\bvec x\|_2^2+\rho\|\bvec x\|_1\right),\quad \of_2(\bvec x)=\left(\frac{\oL}{2}\|\bvec x\|_2^2-\of(\bvec x)+\rho\knorm{\bvec x}_{k,1}\right). %\quad (\text{if $C=\mathbb R^n$})
  \label{eq:l1proxDCdecomp}
\end{align}
The DC decomposition~\eqref{eq:l1proxDCdecomp} gives a closed-form solution for the subproblem \eqref{eq:DCAsubprob}.
%The other DC decomposition~\eqref{eq:l1proxDCdecomp} addresses this issue when $C=\mathbb R^n$.
We call the resulting algorithm the Proximal DC Algorithm (PDCA).
The subproblem~\eqref{eq:DCAsubprob} of PDCA is written as
\begin{align}
\bvec x^{(t+1)}&\in \argmin_{\bvec x\in \mathbb R^n}\left\{ \frac{\oL}{2}\|\bvec x\|_2^2+\rho\|\bvec x\|_1-\bvec x^\top \left(\oL\bvec x^{(t)}-\nabla \of(\bvec x^{(t)})+\bvec s(\bvec x^{(t)}) \right)\right\},\label{eq:l1PDCAsubprob}
\end{align}
where $\bvec s(\bvec x^{(t)})\in\partial \rho\knorm{\bvec x^{(t)}}_{k,1}$.
By using the proximal operator notation:
\begin{align}
\prox_{g}(\bvec u):=\argmin_{\bvec x}\left\{ g(\bvec x)+\frac{1}{2}\|\bvec x-\bvec u\|^{2} \right\},\label{eq:proxoper}
\end{align}
we can further rewrite the subproblem~\eqref{eq:l1PDCAsubprob} as
\begin{align*}
\bvec x^{(t+1)}&\in\argmin_{\bvec x\in \mathbb R^n}\left\{\frac{\rho}{\oL}\|\bvec x\|_1+ \frac{1}{2}\left\|\bvec x-\left(\bvec x^{(t)}-\frac{1}{\oL}\nabla \of(\bvec x^{(t)})+\frac{1}{\oL}\bvec s(\bvec x^{(t)})\right)\right\|_2^2\right\}\nonumber\\
&=\prox_{\frac{\rho}{\oL}\|\cdot\|_1}\left(\bvec x^{(t)}-\frac{1}{\oL}\nabla \of(\bvec x^{(t)})+\frac{1}{\oL}\bvec s(\bvec x^{(t)}) \right),
%\label{eq:PDCAsubprob}
\end{align*}
which is easily computed by using the so-called soft-thresholding~\cite{donoho1994ideal}, whose element is given as
\begin{align}
[\prox_{\frac{\rho}{\oL}\|\cdot\|_{1}}(\bvec u)]_{i}=\sign(u_{i})\max\{u_{i}-\rho/\oL,0\},
\label{eq:softthreshold}
\end{align}
where $\sign(u)=1$ if $u>0$; $-1$ if $u<0$; $0$, otherwise.

%\paragraph{}
%Another DC formulation is given in \cite{thiao2010dc}, which is based on mixed integer programming (MIP).
%They first rewrite $\ell_0$-norm using binary vector $\bvec u$ as
%\begin{align*}
%\|\bvec x\|_0\le k\iff |x_i|\le Mu_i\ (i=1,\ldots,n),\ \bvec 1^\top\bvec u\le k,\ \bvec u\in \{0,1\}^n,
%\end{align*}
%where $M$ is so-called big-$M$ constant, which is set to be sufficiently large.
%Then using the following equivalence:
%\begin{align*}
%\bvec u\in \{0,1\}^n\iff \bvec u\in [0,1]^n, (\bvec 1-\bvec u)^\top \bvec u\le 0.
%\end{align*}
%They finally give a penalized DC formulation of~\eqref{eq:l0const}:
%\begin{align}
%\min_{\bvec x\in C}\left\{f(\bvec x)+\rho(\bvec 1-\bvec u)^\top \bvec u:|x_i|\le Mu_i\ (i=1,\ldots,n),\ \bvec 1^\top\bvec u\le k,\ \bvec u\in [0,1]^n\right\},\label{eq:concavepenalty}
%\end{align}
%which is solved by DCA~\eqref{eq:DCAsubprob}.
%While this approach is originally for Example~\ref{ex:sparsePCA} (and thus an exact penalty is computed only for the example), it works in our general settings.

Bertsimas et al. \cite{bertsimas2016best} addresses \eqref{eq:l0const} without replacing the $\ell_0$-constraint by other terms. 
Since the function $\of%(\bvec x)
$ has a quadratic majorant at each point $\bm{x}^{(t)}$ because of its $L_{\phi}$-smoothness, 
%quadratically majorized by a quadratic and then 
the paper proposes to iteratively solve the subproblems:% the following subproblem is repeatedly solved:
%In their idea, $\of(\bvec x)$ is quadratically majorized and then the following subproblem is repeatedly solved:
\begin{align}
\bvec x^{(t+1)}\in \argmin_{\|\bm{x}\|_0\le k}\left\{ \of(\bvec x^{(t)})-(\bvec x-\bvec x^{(t)})^\top \nabla \of(\bvec x^{(t)})+\frac{\oL}{2}\|\bvec x-\bvec x^{(t)}\|_2^2 \right\}.\label{eq:IHTsubprob}
\end{align}
The subproblem is computed by the so-called hard-thresholding operation, so repeating \eqref{eq:IHTsubprob} is called the Iterative Hard-Thresholding (IHT) algorithm.
They showed that the optimal solution $\hat{\bvec x}$ of $\min_{\bm{x}\in\mathbb R^n}\{\|\bvec x-\bvec u\|_2^2:\|\bvec x\|_0\le k\}$ is obtained as follows: $\hat{\bvec x}$ retains the $k$ largest elements in absolute value of $\bvec u$ and sets the rest elements to zero.
Since the hard-thresholding works only when $C=\mathbb R^n$,
IHT algorithm is not %applied 
applicable to \eqref{eq:l0const} with $C\neq \mathbb R^n$.

\section{DC representation for constrained sparse optimization}\label{sec:representation}
\subsection{Main idea}
%As we see in the previous section, PDCA for~\eqref{eq:l1penalty} works well only when $C=\mathbb R^n$, since its subproblem has no closed-form solutions in general.
PDCA with the DC decomposition \eqref{eq:naiveDCdecomp} for~\eqref{eq:l1penalty} can perform poorly even if a simple convex constraint consists of $C$, since its subproblem has no closed-form solutions in general.
To overcome this issue,
%construct an efficient DCA for sparse problems having some convex constraints,
we give another equivalent DC representation of the $\ell_0$-constraint.
Let us start with the following equivalence results, which slightly generalize Theorem 1 of \cite{gotoh2015dc}. 
\begin{proposition}\label{prop:3_equiv_representations}
Let $\nu:\mathbb{R}\rightarrow\mathbb{R}_+$ be a nonnegative function such that $\nu(a)=0$ if and only if $a=0$, and %let $\bm{\nu}(\bm{x}):=(\nu(x_1),...,\nu(x_n))^\top$, and 
with %an arbitrary 
a permutation $\pi$ of $\{1,\ldots, n\}$, denote by $\nu(x_{\pi(i)})$ the $i$-th largest element of $\nu(x_1),...,\nu(x_n)$, i.e., $\nu(x_{\pi(1)})\geq\cdots\geq\nu(x_{\pi(n)})$. 
 For any integers $k,h$ such that $1\leq{k}<h\leq{n}$, and $\bm{x}\in\mathbb{R}^n$, the following three conditions are equivalent:
\begin{enumerate}
\item $\|\bm{x}\|_0 \leq k$,
\item %$\knorm{\bm{\nu}(\bm{x})}_{h,1}-\knorm{\bm{\nu}(\bm{x})}_{k,1}=0$, (or $\sum\limits_{i=1}^{h}\nu_{(i)}-\sum\limits_{i=1}^{k}\nu_{(i)}=0$), and 
$\sum\limits_{i=1}^{h}\nu(x_{\pi(i)})-\sum\limits_{i=1}^{k}\nu(x_{\pi(i)})=0$, and 
%\item $\sum\limits_{i=1}^n\nu(x_i)-\knorm{\bm{\nu}(\bm{x})}_{k,1}=0$.
\item $\sum\limits_{i=1}^{n}\nu(x_i)-\sum\limits_{i=1}^{k}\nu(x_{\pi(i)})=0$.
%$\|\bm{\nu}(\bm{x})\|_1-\knorm{\bm{\nu}(\bm{x})}_{k,1}=0$, (or $\sum\limits_{i=1}^{n}\nu(x_i)-\sum\limits_{i=1}^{k}\nu_{(i)}=0$),
\end{enumerate}
%where $\nu_{(i)}$ denotes the $i$-th largest element of $\nu(x_1),...,\nu(x_n)$. %, i.e., $\nu(x_{(1)})\geq\cdots\geq\nu(x_{(n)})$. 
 Furthermore, the following three conditions are equivalent:
\begin{enumerate}%[resume]
\setcounter{enumi}{3}
\item $\|\bm{x}\|_0=k$,
\item %$k=\min\{\kappa\in\{1,...,h-1\}:\knorm{\bm{\nu}(\bm{x})}_{h,1}-\knorm{\bm{\nu}(\bm{x})}_{\kappa,1}=0\}$, and
$k=\min\{\kappa\in\{1,...,h-1\}:\sum\limits_{i=1}^{h}\nu(x_{\pi(i)})-\sum\limits_{i=1}^{\kappa}\nu(x_{\pi(i)})=0\}$, and
%\item $k=\min\{k:\sum\limits_{i=1}^n\nu(x_i)-\knorm{\bm{\nu}(\bm{x})}_{k} = 0\}$.
\item %$k=\min\{\kappa\in\{1,...,n-1\}:\|\bm{\nu}(\bm{x})\|_1-\knorm{\bm{\nu}(\bm{x})}_{\kappa,1} = 0\}$.
$k=\min\{\kappa\in\{1,...,n-1\}:\sum\limits_{i=1}^{n}\nu(x_i)-\sum\limits_{i=1}^{\kappa}\nu(x_{\pi(i)})= 0\}$.
\end{enumerate}
\end{proposition}
%Apart from 
If we employ the absolute value for $\nu$, i.e., $\nu(a)=|a|$, it is valid that 
\[
\sum\limits_{i=1}^{h}\nu(x_{\pi(i)})=\knorm{%\bm{\nu}(\bm{x})
(\nu(x_1),...,\nu(x_n))}_{h,1}
\mbox{ and }
\sum\limits_{i=1}^{n}\nu(x_i)=\|%\bm{\nu}(\bm{x})
(\nu(x_1),...,\nu(x_n))\|_1, 
\]
and the above statements result in Theorem 1 of \cite{gotoh2015dc}. 
We can prove Proposition \ref{prop:3_equiv_representations} by just replacing the absolute value with the function $\nu$ in the proof of Theorem 1 of \cite{gotoh2015dc}, %so %we 
and thus 
omit the proof here. 

%Note that if 
%
With $\nu(a)=a^2$ instead of $|a|$, we can attain a quadratic DC representation. 
% as follows:
To align with the notation of $\knorm{\bvec x}_{k,1}$, 
we denote %it
$\knorm{(x_1^2,...,x_n^2)}_{k,1}$ by $\knorm{\bm{x}}_{k,2}^2$.\footnote{In other words, $\knorm{\bm{x}}_{k,2}$ %denotes
equals the $\ell_2$-norm of a subvector composed of top-$k$ elements of $\bvec x$ in %absolute
square value, %by $\knorm{\bvec x}_{k,2}$, 
i.e.,
%\[
$
\knorm{\bm{x}}_{k,2}%:
=\sqrt{x_{\pi(1)}^2+\cdots+x_{\pi(k)}^2}
$ with permutation $\pi$ such that $x_{\pi(1)}^2\geq\cdots\geq x_{\pi(n)}^2$. 
%\]
Analogously to $\knorm{\bm{x}}_{k,1}$, we may call $\knorm{\bm{x}}_{k,2}^2$ {\it top-$(k,2)$ norm}.}
Based on the equivalence between items 1. and 3. in Proposition \ref{prop:3_equiv_representations}, we have another DC representation of the $\ell_0$-constraint:
\begin{align}
\|\bvec x\|_0\le k\iff \|\bvec x\|_2^2-\knorm{\bvec x}_{k,2}^2=0.
\label{eq:quadratic_DC_const}
\end{align}
%defined by \eqref{largeKL1norm},  
%where $\knorm{\bvec x}_{k,2}$, which we call top-$(k,2)$ norm, denotes the $\ell_2$-norm of a subvector composed of top-$k$ elements of $\bvec x$ in absolute value.
Note that $\knorm{\bvec x}_{k,2}^2$ is convex\footnote{More generally, $\knorm{%\bm{\nu}(\bm{x})
(\nu(x_1),...,\nu(x_n))}_{k,1}$ 
(or $\sum_{i=1}^{k}\nu(x_{\pi(i)})$) 
is convex if $\nu$ is convex.} since it can be written as a pointwise maximum of convex functions:
\begin{align*}
\knorm{\bvec x}_{k,2}^2=\max_{\bvec v}\left\{\sum_{i=1}^nv_ix_i^2:\bvec v\in\{0,1\}^n, \|\bvec v\|_1=k\right\}.
\end{align*}
%Using this fact, 
With the equivalence 
\eqref{eq:quadratic_DC_const}, 
we consider the penalized problem associated with~\eqref{eq:l0const}:
\begin{align}
\min_{\bm{x}\in C}\left\{\of(\bvec x)+\rho(\|\bvec x\|_2^2-\knorm{\bvec x}_{k,2}^2)\right\},\label{eq:l2penalty}
\end{align}
where $\rho>0$ denotes a penalty parameter.
The next theorem, which can be proved similarly to Theorem 17.1 in~\cite{nocedal2006numerical}, ensures that problem~\eqref{eq:l2penalty} is essentially equivalent to the original problem~\eqref{eq:l0const} if we take the limit of the penalty parameter $\rho$.
\begin{thm}
Let $\{\rho_t\}$ be an increasing sequence with $\lim_{t\to \infty}\rho_t=\infty$ and $\bvec x_{t}$ be an optimal solution of \eqref{eq:l2penalty} with $\rho=\rho_t$.
Then any %limit
accumulation point $\bvec x^\ast$ of $\{\bvec x_t\}$ is also optimal to \eqref{eq:l0const}.
\end{thm}
\begin{proof}
Let $\bar{\bvec x}$ be an optimal solution of~\eqref{eq:l0const}.
Since $\bvec x_t$ is a minimizer of \eqref{eq:l2penalty} with $\rho=\rho_t$, we have
\begin{align}
\of(\bvec x_t)+\rho_t(\|\bvec x_t\|_2^2-\knorm{\bvec x_t}_{k,2}^2)\le \of(\bar{\bvec x})+\rho_t(\|\bar{\bvec x}\|_2^2-\knorm{\bar{\bvec x}}_{k,2}^2)=\of(\bar{\bvec x}),\label{eq:xt-xbar}
\end{align}
which is transformed into
\begin{align*}
\|\bvec x_t\|_2^2-\knorm{\bvec x_t}_{k,2}^2\le \frac{1}{\rho_t}(\of(\bar{\bvec x})-\of(\bvec x_t)).
\end{align*}
Let $\mathcal T$ be an infinite subsequence such that $\lim_{t\in\mathcal T\to\infty}\bvec x_t=\bvec x^\ast$.
By taking the limit on both sides and considering the nonnegativity of the penalty, we have
\begin{align*}
0\le\|\bvec x^\ast\|_2^2-\knorm{\bvec x^\ast}_{k,2}^2\le \lim_{t\in\mathcal T\to\infty}\frac{1}{\rho_t}(\of(\bar{\bvec x})-\of(\bvec x^\ast))=0,
\end{align*}
which implies $\bvec x^\ast$ is feasible to~\eqref{eq:l0const}.
In addition, by taking the limit on both sides of~\eqref{eq:xt-xbar}, we have
\begin{align*}
\of(\bvec x^\ast)\le \of(\bvec x^\ast)+\lim_{t\to\mathcal T\to\infty}\rho_t(\|\bvec x_t\|_2^2-\knorm{\bvec x_t}_{k,2}^2)\le \of(\bar{\bvec x}).
\end{align*}
Since $\bar{\bvec x}$ is an optimal solution of \eqref{eq:l0const}, $\bvec x^\ast$ is also optimal to \eqref{eq:l0const}.
\end{proof}
%On the other hand, we should note that exact penalty results such as in Theorems 3 and 4 and Corollaries 2 to 4 of \cite{gotoh2015dc} do not hold. 

As we see in the next subsection, %PDCA goes well with problem~\eqref{eq:l2penalty} 
the associated subproblems of the specialized PDCA can be efficiently solved 
owing to the smoothness of $\|\bvec x\|_2^2$.

\subsection{Proximal DC algorithm for the transformed problem}
%While GIST algorithm is not applied to~\eqref{eq:l2penalty} because the proximal operator of $g(\bvec x)=I_C(\bvec x)-\rho\knorm{\bvec x}_{k,2}^2$, is not available, the subproblem of PDCA for~\eqref{eq:l2penalty} is reduced to projection onto $C$:
To apply PDCA to~\eqref{eq:l2penalty}, we consider the following DC decomposition:
\begin{align}
 & \of_1(\bvec x)=\frac{\oL}{2}\|\bvec x\|_2^2+\rho\|\bvec x\|_2^2+I_C(\bvec x),%=
  %\left(\frac{\oL+2\rho}{2}\right) \|\bvec x\|_2^2 +I_C(\bvec x),
   \nonumber\\ %\quad
   &  \of_2(\bvec x)=%\left(
   \frac{\oL}{2}\|\bvec x\|_2^2-\of(\bvec x)+\rho\knorm{\bvec x}_{k,2}^2. %\right),
   \label{eq:l2proxDCdecomp}
\end{align}
%where $I_C(\bvec x)$ denotes the indicator function of $C$.
Then the corresponding PDCA subproblem becomes
\begin{align}
\bvec x^{(t+1)}&\in\argmin_{\bvec x\in \mathbb R^n}\left\{I_C(\bvec x)+ \frac{\oL+2\rho}{2}\left\|\bvec x-\frac{1}{\oL+2\rho}\left(\oL\bvec x^{(t)}-\nabla \of(\bvec x^{(t)})+\bvec s^{(t)}\right)\right\|_2^2\right\}\nonumber\\
&=\prox_{\frac{I_C}{\oL+2\rho}}\left(\frac{1}{\oL+2\rho}\left(\oL\bvec x^{(t)}-\nabla \of(\bvec x^{(t)})+\bvec s(\bvec x^{(t)})\right) \right),
\label{eq:l2PDCAsubprob}
\end{align}
where $\bvec s^{(t)}\in \partial (\rho\knorm{\bvec x}_{k,2}^2)$.
The subdifferential of %top-$(k,2)$ norm
$\knorm{\bm{x}}_{k,2}^2$ is given as
\begin{align*}
%\partial(\knorm{\bvec x}_{k,2}^2)=\left\{ \bvec v: v_{(i)}=
%\begin{cases}
%2x_{(i)} &(i\le k)\\
%0 & %(\pi(i)>k)
%(\mbox{otherwise})
%\end{cases}
%\right\}.
%\end{align*}
%where $\pi$ is an arbitrary permutation of $\{1,\ldots, n\}$ that satisfies $|x_{\pi(1)}|\ge \cdots \ge |x_{\pi(n)}|$.
\partial(\knorm{\bvec x}_{k,2}^2)=\left\{ \bvec v: v_i=
\begin{cases}
2x_i &(\pi(i)\le k)\\
0 & (\pi(i)>k)
\end{cases}
%\right\},
\right\}.
\end{align*}
%where $\pi$ is an arbitrary permutation of $\{1,\ldots, n\}$ that satisfies $|x_{\pi(1)}|\ge \cdots \ge |x_{\pi(n)}|$.
Note that the proximal operator of $I_C$ is nothing but the projection onto $C$.
%, which is denoted by $\proj_C$.
Therefore, the subproblem~\eqref{eq:l2PDCAsubprob} is easily solved for various feasible sets $C$.
We list below how to obtain $\proj_C$ for the three constraint sets in Examples~\ref{ex:sparsePCA}--\ref{ex:sparsennreg}.
\begin{enumerate}[{(}i{)}]
\item For $C=\{\bvec x\in\mathbb R^n:\|\bvec x\|_2\le 1\}$, $\proj_C$ is given by
\begin{align*}
\proj_C(\bvec u)=
\begin{cases}
\frac{\bvec u}{\|\bvec u\|_2}&(\|\bvec u\|_2\ge 1),\\
\bvec u&(\|\bvec u\|_2<1).
\end{cases}
\end{align*}
\item
For $C=\{\bvec x\in\mathbb R^n:\bvec 1^\top \bvec x =1\}$, $\proj_C$ is given by
\begin{align*}
\proj_C(\bvec u)=\bvec u+(1-\bvec 1^\top \bvec u)\bvec 1/n.
\end{align*}
\item
For $C=\{\bvec x\in\mathbb R^n: x_i\ge 0\ (i\in I\subseteq \{1,\ldots,n\})\}$, $\proj_C(\bvec u)$ is obtained by setting the negative elements of $\bvec u$ corresponding to $I$ to zero and retaining the rest.
\end{enumerate}

We summarize the procedure of PDCA for the transformed problem~\eqref{eq:l2penalty} in Algorithm~\ref{alg:PDCA}.
For practical use, the termination criterion of Algorithm~\ref{alg:PDCA} is replaced by
%$\of(\bvec x^{(t)}) - \of(\bvec x^{(t+1)}) < \epsilon$ using a sufficiently small positive value $\epsilon$.
$\Phi(\bvec x^{(t)}) - \Phi(\bvec x^{(t+1)}) < \varepsilon$, where $\Phi$ denotes the objective function in~\eqref{eq:l2penalty} and $\varepsilon$ is a sufficiently small positive value.
As we mentioned in Section~\ref{sec:preliminaries}, Algorithm~\ref{alg:PDCA} is a kind of DCA with the special DC decomposition.
Since the global convergence of DCA is shown in \cite{pham1997convex} for a general problem setting including~\eqref{eq:l2penalty}, the convergence property is also valid for Algorithm~\ref{alg:PDCA}.
\begin{thm}
Let $\{\bvec x^{(t)}\}$ be the sequence generated by Algorithm~\ref{alg:PDCA}. Then $\{\bvec x^{(t)}\}$ globally converges to a stationary point $\bvec x^\ast\in C$ of \eqref{eq:l2penalty}, i.e.,
\begin{align*}
\bvec 0\in \partial \phi_1(\bvec x^\ast)-\partial\phi_2(\bvec x^\ast),
\end{align*}
where $\phi_1(\bvec x)$ and $\phi_2(\bvec x)$ are given by~\eqref{eq:l2proxDCdecomp}.
\end{thm}

\begin{algorithm}\caption{Proximal DC Algorithm (PDCA) for~\eqref{eq:l2penalty}}\label{alg:PDCA}
\begin{algorithmic}
\STATE $\bvec x^{(0)}\in C$.
\FOR{$t=0,1,\ldots$}
\STATE Pick a subgradient $\bvec s(\bvec x^{(t)})\in\partial g_2(\bvec x^{(t)})$ and compute
\STATE $\bvec x^{(t+1)}=\proj_{C}\left(\frac{1}{\oL+2\rho}\left(\oL\bvec x^{(t)}-\nabla \of(\bvec x^{(t)})+\bvec s(\bvec x^{(t)})\right) \right)$.
\ENDFOR
\end{algorithmic}
\end{algorithm}

%This procedure generates a sequence $\{\bvec x^{(k)}\}$ that converges to a stationary point, retaining $\mathrm{O}(1/k^2)$ convergence rate for the convex case.
%In addition, they provide a detailed convergence analysis under KL property using the monotonicity of the sequence $\{F(\bvec x^{(k)})\}$.
%Furthermore, they proposed nonmonotone APG in order to reduce the computation cost at each iteration.
%In applying GIST or APG for~\eqref{DCreg}, how to compute the nonconvex proximal operator~\eqref{proxoper} is crucial.
%Although \cite{gong2013general} shows several examples of sparse regularizers that have closed-form solutions for \eqref{proxoper},
%that is not the case with our setting where $g(\bvec x)=I_C(\bvec x)-\rho\knorm{\bvec x}_{k,2}^2$.
%In the next subsection, we will propose a modified APG, which avoids nonconvex proximal operation and is widely applicable, retaining the same convergence property as the nonmonotone APG.

%\section{Link between proximal DC algorithm and proximal gradient method}
\section{PDCA with backtracking step size rule}\label{sec:link}
In this section, we show %the 
a link between the Proximal DC Algorithm (PDCA) and 
the Proximal Gradient Method (PGM), %\cite{}, 
%and thereby generalize the proximal DC algorithm,
%which leads to
and present a 
PDCA with backtracking step size rule. % and its accelerated variant (the latter of which will be shown in the next section).
More specifically, we first discuss that the framework of PDCA can be extended to more general settings.
Then we clarify that PDCA is a generalized version of PGM for DC optimization, which implies that some useful techniques to speed up PGM can 
also be employed in PDCA. %too.

\subsection{Proximal DC algorithm for composite nonconvex optimization}
We consider the following composite nonconvex optimization problem:
\begin{align}
\min_{\bm{x}\in\mathbb R^{n}}F(\bvec x):=f(\bvec x)+g(\bvec x),\label{eq:DCreg}
\end{align}
where $f,g:\mathbb R^{n}\to\mathbb R\cup \{+\infty\}$.
We make the following assumptions on~\eqref{eq:DCreg}.
\begin{assumption}\label{assump_composite}
\begin{enumerate}[{(}a{)}]
\item $f(\bvec x)$ is continuously differentiable with $L$-Lipschitz continuous gradient.
\item  $g(\bvec x)$ is decomposed into a DC function as
\begin{align}
g(\bvec x)=g_{1}(\bvec x)-g_{2}(\bvec x),
\label{eq:g=g1-g2}
\end{align}
where $g_{1}(\bvec x)$ is proper, lower semicontinuous and convex, and $g_{2}(\bvec x)$ is continuous and convex.
\item  $F(\bvec x)$ is bounded from below and coercive.
\end{enumerate}
\end{assumption}
The penalized formulation \eqref{eq:l2penalty} of the sparse constrained problem can be regarded as problem~\eqref{eq:DCreg} satisfying Assumption~\ref{assump_composite};
\begin{itemize}
  \item
$\of(\bvec x)+\rho\|\bvec x\|_2^2$ corresponds to the smooth term
    $f(\bvec x)$ in \eqref{eq:DCreg} which has $L=\oL+2\rho$,
\item
  the indicator function $I_C(\bvec x)$ of $C$ corresponds to $g_1(\bvec x)$, and
\item
  $g_2(\bvec x)=\rho\knorm{\bvec x}_{k,2}^2$.
  \end{itemize}
In addition, many nonconvex regularized problems are included in this setting, as %seen 
shown in Table~\ref{table:regularizer}.
\begin{table}[h]
\begin{center}
\caption{DC decompositions of sparse regularizers ($\lambda$ and $\theta$ denote nonnegative hyperparameters), given in~\cite{gong2013general} except for $\ell_{1-2}$.}\label{table:regularizer}
{\small
\begin{tabular}{|l||c|c|}\hline
\multicolumn{1}{|c||}{name of regularizer}&$g_1$& $g_2$\\\hline
$\ell_{1}$ norm~\cite{tibshirani1996regression} & $\lambda\|\bvec x\|_{1}$ & 0 \\
%ridge &  $\lambda\|\bvec x\|_{2}^{2}$ & 0 \\
capped-$\ell_{1}$~\cite{zhang2010analysis} &  $\lambda\|\bvec x\|_{1}$ & $\lambda\sum\limits_{i=1}^{n} \max\{|x_{i}| -\theta,0\}$ \\
LSP (Log Sum Penalty)~\cite{candes2008enhancing} & $\lambda\|\bvec x\|_{1}$ & $\lambda\sum\limits_{i=1}^{n}(|x_{i}|-\log(1+|x_{i}|/\theta)) $ \\
\begin{tabular}{@{}l}SCAD (Smoothly Clipped\\~~~~~Absolute Deviation)~\cite{fan2001variable}\end{tabular} & $\lambda\|\bvec x\|_{1}$ & $
\sum\limits_{i=1}^{n} \left\{
\begin{array}{ll}
  0  & \mbox{ if } |x_i|\leq \lambda\\
  \frac{x_i^2-2\lambda |x_i|+\lambda^2}{2(\theta-1)} & \mbox{ if } \lambda < |x_i|\leq \theta \lambda \\
   (\lambda |x_i|-\frac{(\theta+1)\lambda^2}{2})  & \mbox{ if } |x_i| > \theta \lambda \\
\end{array}
\right.
$  \\
\begin{tabular}{@{}l}MCP (Minimax\\~~~~~Concave Penalty)~\cite{zhang2010nearly}\end{tabular} & $\lambda\|\bvec x\|_{1}$ & $
\sum\limits_{i=1}^{n}
\left\{
\begin{array}{ll}
  \frac{x_i^2}{2\theta} & \mbox{ if }  |x_i|\leq \theta \lambda  \\
  \lambda |x_i| -\frac{\theta \lambda^2}{2}  & \mbox{ if }  |x_i|> \theta \lambda  \\
\end{array}
\right.
$    \\
$\ell_{1-2}$~\cite{yin2015minimization} & $\lambda\|\bvec x\|_{1}$ & $\lambda\|\bvec x\|_{2}$ \\\hline
\end{tabular}
}
\end{center}
\end{table}

We can naturally extend our PDCA to~\eqref{eq:DCreg}, which is originally proposed for~\eqref{eq:l1penalty} in~\cite{gotoh2015dc}.
Similarly to~\eqref{eq:l1proxDCdecomp}, we consider the following DC decomposition of $F$:
\begin{align}
F(\bvec x)=\left(\frac{L}{2}\|\bvec x\|_2^2+g_1(\bvec x)\right)- \left(\frac{L}{2}\|\bvec x\|_2^2-f(\bvec x)+g_2(\bvec x)\right).\label{eq:proxDCdecomp}
\end{align} 
Then the subproblem of DCA for~\eqref{eq:proxDCdecomp} becomes
\begin{align}
\bvec x^{(t+1)}&\in \argmin_{\bvec x\in \mathbb R^n}\left\{ \frac{L}{2}\|\bvec x\|_2^2+g_1(\bvec x)-\bvec x^\top \left(L\bvec x^{(t)}-\nabla f(\bvec x^{(t)})+\bvec s(\bvec x^{(t)}) \right)\right\}\nonumber\\
&=\prox_{g_1/L}\left(\bvec x^{(t)}-\frac{1}{L}\nabla f(\bvec x^{(t)})+\frac{1}{L}\bvec s(\bvec x^{(t)}) \right),
\label{eq:PDCAsubprob}
\end{align}
where $\bvec s^{(t)}\in\partial g_2(\bvec x^{(t)})$.
%Since
The subproblem~\eqref{eq:PDCAsubprob} of PDCA is reduced to calculating the proximal operator of $g_1/L$,
which leads to closed-form solutions for various $g_1$.
%it is easily computed for a wide class of $g_1$.

Now we recall the Sequential Convex Programming (SCP)~\cite{lu2012sequential} as a related work.
SCP solves problem~\eqref{eq:DCreg} by generating a sequence $\{\bvec x^{(t)}\}$ obtained via
\begin{align*}
&\bvec x^{(t+1)}\in \argmin_{\bvec x\in\mathbb R^{n}}f(\bvec x^{(t)})+\langle \nabla f(\bvec x^{(t)}),\bvec x-\bvec x^{(t)} \rangle\nonumber\\
&+\frac{L}{2}\|\bvec x-\bvec x^{(t)}\|_{2}^{2}+g_{1}(\bvec x)-g_{2}(\bvec x^{(t)})-\langle s(\bvec x^{(t)}),\bvec x-\bvec x^{(t)} \rangle. %\label{SCP}
\end{align*}
This problem is essentially the same as \eqref{eq:PDCAsubprob}, but the paper does not mention how to solve such convex subproblems, nor the way of computing closed-form solutions. 
They derive this algorithm and analyze its convergence independently of the theory of DC programming.
Our algorithm sheds a new light on SCP. Namely, SCP can be %seen 
viewed as a variant of DC algorithm and thus its convergence property such as global convergence is automatically satisfied.

%\subsection{Relation to PGM for convex optimization}
\subsection{Relation to PGM variants}
Especially for convex functions $f$ and $g$, we can see that PDCA reduces to the renowned Proximal Gradient Method (PGM):
\begin{align}
\bvec x^{(t+1)}=\prox_{g_1/L}\left(\bvec x^{(t)}-\frac{1}{L}\nabla f(\bvec x^{(t)}) \right).
\label{eq:PGMsubprob}
\end{align}

PGMs for convex optimization problems have been investigated in a different research stream from DCA
for nonconvex optimization problems, but we can find a similarity of the resulting
subproblems: \eqref{eq:PDCAsubprob} and \eqref{eq:PGMsubprob}.
In recent years, developing efficient algorithms for solving convex cases of
%the importance of taking advantage of the structure of convex optimization problems like
\eqref{eq:DCreg} has become a topic of intense research especially
in the machine learning community and  various techniques
for obtaining faster convergence were proposed for PGMs. We also can use 
%To achieve faster convergence, several
such techniques including the backtracking and the acceleration for our method, namely, PDCA.
%have been proposed for PGM and its variants,
%Especially if $f$ and $g$ in~\eqref{eq:DCreg} are convex, \eqref{PGM} reduces to the proximal gradient method (PGM), which yields a global optimum and whose convergence rate is guaranteed to be $\mathrm{O}(1/k)$.

Currently, popular research directions regarding PGMs include applying PGMs
to nonconvex optimization problems.
For example, General Iterative-Shrinkage Thresholding (GIST) algorithm~\cite{gong2013general}
was proposed for nonconvex \eqref{eq:DCreg}.
GIST generates a sequence $\{\bvec x^{(t)}\}$ by
\begin{align}
\bvec x^{(t+1)}=\prox_{g/l^{(t)}}\left(\bvec x^{(t)}-\frac{1}{l^{(t)}}\nabla f(\bvec x^{(t)})\right),\label{PGM}
\end{align}
where $1/l^{(t)}$ is a proper step size. The paper \cite{gong2013general}
showed closed-form solutions of \eqref{PGM} for the regularizers in Table~\ref{table:regularizer} except for $\ell_{1-2}$
\footnote{%.
For the $\ell_{1-2}$ regularizer, Liu and Pong~\cite{liu2016further} showed closed-form solutions of~\eqref{PGM}.
}
.

%some examples, e.g., 
%$g$ is expressed by the summation
%of separable functions as shown in Table~\ref{table:regularizer}.
Note that applying GIST to the reformulations \eqref{eq:l1penalty} and \eqref{eq:l2penalty} for constrained sparse optimization problems \eqref{eq:l0const}
seems difficult
because of the term $I_C(\bvec x)$.
The sequence $\{\bvec x^{(t)}\}$ of GIST subsequentially converges to a stationary point of~\eqref{eq:DCreg}, as far as $l^{(t)}$ is fixed to an arbitrary value larger than a Lipschitz constant $L$ of $\nabla f(\bvec x)$.
We will %write 
describe how to determine $l^{(t)}$ in practice, later when %showing 
elaborating on our method.

\subsection{Backtracking}
To achieve faster convergence, several techniques %including 
such as the backtracking and the acceleration have been proposed for PGM and its variants, the latter of which is mentioned in Section~\ref{sec:acceleration}.
The backtracking line search initialized by Barzilai-Borwein (BB) rule~\cite{barzilai1988two}  is employed in GIST~\cite{gong2013general}
 to use a larger step size $1/l^{(t)}$ instead of $1/L$.
In the backtracking, we accept $l^{(t)}$ if the following criterion is satisfied for $\sigma\in(0,1)$:
\begin{align}\label{eq:backtracking}
F(\bvec x^{(t+1)})\le F(\bvec x^{(t)})-\frac{\sigma}{2}\|\bvec x^{(t+1)}-\bvec x^{(t)}\|_2^2,
\end{align}
otherwise $l^{(t)}\leftarrow \eta l^{(t)}$ with $\eta>1$
and check the above inequality again. %, where $\sigma\in(0,1)$ and $\eta>1$.
The initial $l^{(t)}$ at each iteration $t$ is given by
the BB rule~\cite{barzilai1988two} as %is a strategy to properly initialize $l^{(t)}$ at each iteration $t$:
\begin{align}\label{eq:BBrule}
l^{(t)}=\frac{\langle \bvec x^{(t)}-\bvec x^{(t-1)},\bvec x^{(t)}-\bvec x^{(t-1)} \rangle}{\langle \bvec x^{(t)}-\bvec x^{(t-1)},\nabla f(\bvec x^{(t)})-\nabla f(\bvec x^{(t-1)})\rangle},
\end{align}
For convergence and practical use, $l^{(t)}$ is projected onto the interval $[l_{\min},l_{\max}]$ with $0<l_{\min}<l_{\max}$.

Now we consider employing the backtracking technique in PDCA.
We use a larger step size $1/l^{(t)}$ instead of $1/L$:
\begin{align}
\bvec x^{(t+1)}&\in \argmin_{\bvec x\in \mathbb R^n}\left\{ \frac{l^{(t)}}{2}\|\bvec x\|_2^2+g_1(\bvec x)-\bvec x^\top \left(l^{(t)}\bvec x^{(t)}-\nabla f(\bvec x^{(t)})+\bvec s(\bvec x^{(t)}) \right)\right\}\nonumber\\
&=\prox_{g_1/l^{(t)}}\left(\bvec x^{(t)}-\frac{1}{l^{(t)}}\nabla f(\bvec x^{(t)})+\frac{1}{l^{(t)}}\bvec s(\bvec x^{(t)}) \right),
\label{eq:PDCAbtsubprob}
\end{align}
The resulting algorithm is summarized in Algorithm~\ref{alg:PDCAbt}, whose convergence is analyzed similarly to GIST algorithm in~\cite{gong2013general}.
\begin{algorithm}\caption{Proximal DC Algorithm for~\eqref{eq:DCreg} with backtracking line search}\label{alg:PDCAbt}
\begin{algorithmic}
\STATE  $\bvec x^{(0)}\in \dom\ g_1$.
\FOR{$t=0,1,\ldots$}
\STATE Compute $\bvec x^{(t+1)}$ by~\eqref{eq:PDCAbtsubprob}, where the step size $1/l^{(t)}$ is dynamically computed by~\eqref{eq:backtracking}--\eqref{eq:BBrule}.
\ENDFOR
\end{algorithmic}
\end{algorithm}

\begin{thm}\label{theo:PDCA_conv}
The sequence $\{\bvec x^{(t)}\}$ generated by Algorithm~\ref{alg:PDCAbt} converges to a stationary point of \eqref{eq:DCreg}.
\end{thm}
See Appendix~\ref{sec:proof:PDCA_conv} for the Proof of Theorem~\ref{theo:PDCA_conv}.
The next theorem ensures the convergence rate of Algorithm~\ref{alg:PDCAbt} with respect to $\|\bvec x^{(t+1)}-\bvec x^{(t)}\|_2^2$.
The proof is almost the same as Theorem 2 in~\cite{gong2013general}.
\begin{thm}\label{theo:PDCA_speed}
For the sequence $\{\bvec x^{(t)}\}$ generated by Algorithm~\ref{alg:PDCAbt} and
its accumulation point $\bvec x^{\ast}$, the following holds for any $\tau\ge 1$;
\begin{align*}
\min_{0\le t\le \tau}\|\bvec x^{(t+1)}-\bvec x^{(t)}\|_2^2\le \frac{2(F(\bvec x^{(0)})-F(\bvec x^{\ast}))}{\tau\sigma}.
\end{align*}
\end{thm}
\begin{proof}
It follows from the criterion~\eqref{eq:backtracking} that
\begin{align*}
\frac{\sigma}{2}\|\bvec x^{(t+1)}-\bvec x^{(t)}\|_2^2\le  F(\bvec x^{(t)})- F(\bvec x^{(t+1)}).
\end{align*}
Summing the above inequality over $t=0,\ldots,\tau$, we have
\begin{align*}
\frac{\sigma}{2}\sum_{t=0}^\tau\|\bvec x^{(t+1)}-\bvec x^{(t)}\|_2^2\le  F(\bvec x^{(0)})- F(\bvec x^{(t+1)}).
\end{align*}
Thus we have
\begin{align*}
\min_{0\le t\le n}\|\bvec x^{(t+1)}-\bvec x^{(t)}\|_2^2\le  \frac{2(F(\bvec x^{(0)})- F(\bvec x^{(t+1)}))}{\tau\sigma}\le  \frac{2(F(\bvec x^{(0)})- F(\bvec x^{\ast}))}{\tau\sigma}.
\end{align*}
\end{proof}

\section{Accelerated algorithm for constrained sparse optimization}\label{sec:acceleration}
In this section, we provide an accelerated version of PDCA for~\eqref{eq:DCreg}.

\subsection{Overview of accelerated methods for nonconvex optimization}
For convex $f$ and $g$, the so-called Nesterov's acceleration technique helps PGM accelerate 
practically and theoretically; the resulting method is known as Accelerated Proximal Gradient (APG) method~\cite{beck2009fast}.
APG is guaranteed to have $\mathrm{O}(1/t^2)$ convergence rate, which is optimal among all the first-order methods.
However, APG in \cite{beck2009fast} is for convex problems and has no guarantees to yield stationary points for the nonconvex case until quite recently.

%Recently, Li and Lin~\cite{li2015accelerated} has proposed two APGs, each of which uses a different type of acceleration,
% for nonconvex optimization: monotone one and nonmonotone one.
Recently, Li and Lin~\cite{li2015accelerated} has proposed two APGs for nonconvex optimization: monotone APG and nonmonotone APG, each of which uses a different type of acceleration. 
Their nonmonotone APG with fixed step size is summarized in Algorithm~\ref{alg:nmAPG}.
This procedure generates a sequence $\{\bvec x^{(t)}\}$ that subsequentially converges to a stationary point, retaining $\mathrm{O}(1/t^2)$ convergence rate for the convex case.
%In addition, they provide a detailed convergence analysis under KL property using the monotonicity of the sequence $\{F(\bvec x^{(k)})\}$.
%Furthermore, they proposed nonmonotone APG in order to reduce the computation cost at each iteration.
\begin{algorithm}\caption{nonmonotone Accelerated Proximal Gradient (nm-APG) method~\cite{li2015accelerated}}\label{alg:nmAPG}
\begin{algorithmic}
\STATE $\bvec x^{(0)}=\bvec x^{(1)}=\bvec z^{(1)}\in\dom\ g_1, \theta^{(0)}=0$, $\theta^{(1)}=1$, $\delta>0$, and $\eta\in(0,1]$.
\FOR{$t=1,2,\ldots$}
\STATE $\bvec y^{(t)}=\bvec x^{(t)}+\frac{\theta^{(t-1)}}{\theta^{(t)}}(\bvec z^{(t)}-\bvec x^{(t)})+\frac{\theta^{(t-1)}-1}{\theta^{(t)}}(\bvec x^{(t)}-\bvec x^{(t-1)})$,
\STATE $\bvec z^{(t+1)}=\prox_{g/L}(\bvec y^{(t)}-\frac{1}{L}\nabla f(\bvec y^{(t)}))$.
\IF{$F(\bvec z^{(t+1)})+\delta \|\bvec z^{(t+1)}-\bvec y^{(t)}\|_2^2\le  \frac{\sum_{j=1}^t\eta^{t-j}F(\bvec x^{(j)})}{\sum_{j=1}^t \eta^{t-j}}$}
\STATE $\bvec x^{(t+1)}=\bvec z^{(t+1)}$.
\ELSE
\STATE $\bvec v^{(t+1)}=\prox_{g/L}(\bvec x^{(t)}-\frac{1}{L}\nabla f(\bvec x^{(t)}))$.
\ENDIF\STATE $\theta^{(t+1)}=\frac{\sqrt{4(\theta^{(t)})^2+1}+1}{2}$.
\ENDFOR
\end{algorithmic}
\end{algorithm}

We can take advantage of the similarity of PDCA and PGM for developing
an accelerated version of PDCA, as we did for developing
PDCA with backtracking step size rule in the previous section.
The accelerated version of PDCA, which %is
will be discussed in Section~\ref{sec:acceleratedPDCA},
%utilized
utilizes the acceleration technique of Li and Lin~\cite{li2015accelerated}. 

%\subsubsection{Proximal DC algorithm with extrapolation}
Quite recently, for the case where $f$ is convex, Wen et al.~\cite{wen2016proximal} has developed a modified PDCA by adding an extrapolation step to speed up its convergence.
They call it the proximal Difference-of-Convex Algorithm with extrapolation (pDCA${}_e$), which is summarized in Algorithm~\ref{alg:pDCAe}.
\begin{algorithm}\caption{Proximal DC Algorithm with extrapolation (pDCA${}_e$)~\cite{wen2016proximal}}\label{alg:pDCAe}
\begin{algorithmic}
\STATE $\bvec x^{(0)}=\bvec x^{(1)}\in\dom\, g_1$, $\{\beta_t\}\subseteq[0,1)$ with $\sup_t \beta_t<1$.
\FOR{$t=1,2,\ldots$}
\STATE Pick any $\bvec s(\bvec x^{(t)})\in\partial g_2(\bvec x^{(t)})$ and compute
\STATE $\bvec y^{(t)}=\bvec x^{(t)}+\beta_t(\bvec x^{(t)}-\bvec x^{(t-1)})$,
\STATE $\bvec x^{(t+1)}=\prox_{g_1/L}\left(\bvec y^{(t)}-\frac{1}{L}\nabla f(\bvec y^{(t)})+\frac{1}{L}\bvec s^{(t)} \right)$.
\ENDFOR
\end{algorithmic}
\end{algorithm}
We can see that pDCA${}_e$ is general enough to include many algorithms.
It reduces to PDCA for convex $f$ by setting $\beta_t\equiv 0$ in Algorithm~\ref{alg:pDCAe}, and to FISTA with the fixed or adaptive restart~\cite{o2015adaptive} for convex $f$ and $g$ by choosing $\{\beta_t\}$ appropriately.
 In other words, Wen et al.~\cite{wen2016proximal} developed another type of acceleration for PDCA,
while we utilized the acceleration technique of Li and Lin~\cite{li2015accelerated} for accelerating PDCA. 
Both works were done in parallel at almost the same time, and we added comparison of two acceleration
methods to our numerical experiment.

\subsection{Proposed Algorithm} \label{sec:acceleratedPDCA}
We propose the Accelerated Proximal DC Algorithm (APDCA) for~\eqref{eq:DCreg} by applying %.
%In our algorithm, the concave term $-g_2(\bvec x)$ and then computes the proximal operator of $g_2(\bvec x)$:
%\begin{align}
%\bvec x^{(k+1)}=\prox_{g_{1}/l^{(k)}}\left(\bvec x^{(k)}-\frac{1}{l^{(k)}}\left(\nabla f(\bvec x^{(k)})- s(\bvec x^{(k)})\right)\right),\label{PDCA}
%\end{align}
%where $s(\bvec x^{(k)})$ is a subgradient of $g_{2}$ at $\bvec x^{(k)}$, i.e.,
%\begin{align*}
%s(\bvec x^{(k)})\in \partial g_{2}(\bvec x^{(k)}):= \{\bvec y:g_{2}(\bvec x)\ge g_{2}(\bvec x^{(k)})+\langle \bvec y, \bvec x-\bvec x^{(k)}\rangle \ (\bvec x\in \mathbb R^{n}) \}.
%\end{align*}
%We apply 
the Nesterov's acceleration technique to PDCA.
In order to establish good convergence properties, we employ the techniques used in Algorithm~\ref{alg:nmAPG}.
%First of all, we accelerate PDCA retaining the monotonicity of the objective function.
The procedure is summarized in Algorithm~\ref{alg:APDCA}, using the following procedure:
\begin{align}
\bvec y^{(t)}&=\bvec x^{(t)}+\frac{\theta^{(t-1)}}{\theta^{(t)}}(\bvec z^{(t)}-\bvec x^{(t)})+\frac{\theta^{(t-1)}-1}{\theta^{(t)}}(\bvec x^{(t)}-\bvec x^{(t-1)}),\label{eq:APDCAtop}\\
\bvec z^{(t+1)}&=\prox_{g_1/l_y^{(t)}}\left( \bvec y^{(t)}-\frac{1}{l_y^{(t)}}\nabla f(\bvec y^{(t)})+\frac{1}{l_y^{(t)}}\bvec s(\bvec y^{(t)}) \right),\label{eq:computez}\\
\bvec v^{(t+1)}&=\prox_{g_1/l_x^{(t)}}\left( \bvec x^{(t)}-\frac{1}{l_x^{(t)}}\nabla f(\bvec x^{(t)})+\frac{1}{l_x^{(t)}}\bvec s(\bvec x^{(t)}) \right),\label{eq:computev}\\
%\theta^{(t+1)}&=\frac{\sqrt{4(\theta^{(t)})^2+1}+1}{2}, \label{eq:APDCAupdateT}\\
\bvec x^{(t+1)}&=\begin{cases} 
\bvec z^{(t+1)},&(F(\bvec z^{(t+1)})\le F(\bvec v^{(t+1)})),\\
\bvec v^{(t+1)},&(\text{otherwise}),
\end{cases}\label{eq:APDCAbottom} \\
\theta^{(t+1)}&=\frac{\sqrt{4(\theta^{(t)})^2+1}+1}{2}, \label{eq:APDCAupdateT}
\end{align}
where $\bvec s(\bvec x^{(t)})$ and $\bvec s(\bvec y^{(t)})$ denote subgradients of $g_{2}$ at $\bvec x^{(t)}$ and $\bvec y^{(t)}$, respectively.

\begin{algorithm}\caption{Accelerated Proximal DC Algorithm (APDCA)}\label{alg:APDCA}
\begin{algorithmic}
\STATE $\bvec x^{(0)}=\bvec x^{(1)}=\bvec z^{(1)}\in\dom\ g_1, \theta^{(0)}=0,$ and $\theta^{(1)}=1$.
\FOR{$t=1,2,\ldots$}
\STATE Compute $\bvec y^{(t)}$ and $\bvec z^{(t+1)}$ by~\eqref{eq:APDCAtop}--\eqref{eq:computez}, where 
the step size $1/l_y^{(t)}$ is fixed smaller than $1/L$ or dynamically computed by~\eqref{eq:backtracking}--\eqref{eq:BBrule}.
\IF{\eqref{eq:nmcriteria} holds}
\STATE $\bvec x^{(t+1)}=\bvec z^{(t+1)}$.
\ELSE
\STATE Compute $\bvec x^{(t+1)}$ by~\eqref{eq:computev} and \eqref{eq:APDCAbottom}, where 
the step size $1/l_x^{(t)}$ is fixed smaller than $1/L$ or dynamically computed by~\eqref{eq:backtracking} and \eqref{eq:BBrulefornonmonotone}.
\ENDIF\STATE Compute $\theta^{(t+1)}$ by \eqref{eq:APDCAupdateT}.
\ENDFOR
\end{algorithmic}
\end{algorithm}

%We follow~\cite{li2015accelerated} to
By following \cite{li2015accelerated}, we 
take $\bvec y^{(t)}$ as a good extrapolation and omit to compute the second proximal operator~\eqref{eq:computev}  if the following criterion given by \cite{zhang2004nonmonotone} is satisfied:
\begin{align}
F(\bvec z^{(t+1)})+\delta \|\bvec z^{(t+1)}-\bvec y^{(t)}\|_2^2\le c^{(t)}:= \frac{\sum_{j=1}^t\eta^{t-j}F(\bvec x^{(j)})}{\sum_{j=1}^t \eta^{t-j}},\label{eq:nmcriteria}
\end{align}
where $\delta>0$ and $\eta\in(0,1]$ controls the weights of the convex combination.
Note that $c^{(t)}$ is computed step by step as
\begin{align*}
q^{(t+1)}&=\eta q^{(t)}+1,\\
c^{(t+1)}&=\frac{\eta q^{(t)}+F(\bvec x^{(t+1)})}{q^{(t+1)}},
\end{align*}
with $q^{(1)}=1$ and $c^{(1)}=F(\bvec x^{(0)})$.
Since $\bvec v^{(t)}$ is not necessarily computed in each iteration, the following initialization rule  for $l_x^{(t)}$ is used instead of~\eqref{eq:BBrule}:
\begin{align}\label{eq:BBrulefornonmonotone}
l_x^{(t)}=\frac{\langle \bvec x^{(t)}-\bvec y^{(t-1)},\bvec x^{(t)}-\bvec y^{(t-1)} \rangle}{\langle \bvec x^{(t)}-\bvec y^{(t-1)},\nabla f(\bvec x^{(t)})-\nabla f(\bvec y^{(t-1)})\rangle},
\end{align}
after which we project $l_x^{(t)}$ onto $[l_{\min},l_{\max}]$.

The convergence of Algorithm~\ref{alg:APDCA} is guaranteed by the next theorem,
which is proved similarly to Theorem 4 in~\cite{li2015accelerated}. 
%whose proof is partially the same as that of Theorem 4 in~\cite{li2015accelerated} and thus omitted here. 
\begin{thm}\label{theo:nmPDCA_conv}
%\footnote{We provide this theorem without proof because 
%it can be proved with 
%a slight modification of the proof of Theorem 4 in~\cite{li2015accelerated}.
%To see how to modify, see the difference of the proof of Theorem~\ref{theo:PDCA_conv} with that of Theorem 1 in~\cite{gong2013general}, which is mentioned in the supplementary material (Section~B).}
Let $\Omega_1$ be the set of every $t$ at which \eqref{eq:nmcriteria} is satisfied and $\Omega_2$ be the set of the rest.
Then the sequences $\{\bvec x^{(t)}\}$, $\{\bvec v^{(t)}\}$, and $\{\bvec y^{(t_1)}\}_{t_1\in \Omega_1}$ generated by Algorithm~\ref{alg:APDCA} are bounded and
\begin{enumerate}
\item if $\Omega_1$ or $\Omega_2$ is finite, then any %limit 
accumulation point $\bvec x^\ast$ of $\{\bvec x^{(t)}\}$ is a stationary point of~\eqref{eq:DCreg};
\item
otherwise, any %limit 
accumulation points $\bvec x^\ast$ of $\{\bvec x^{(t_1+1)}\}_{t_1\in\Omega_1}$, $\bvec y^\ast$ of $\{\bvec y^{(t_1)}\}_{t_1\in\Omega_1}$, $\bvec v^\ast$ of $\{\bvec v^{(t_2+1)}\}_{t_2\in\Omega_2}$, and $\bvec x^\ast$ of $\{\bvec x^{(t_2)}\}_{t_2\in\Omega_2}$ are stationary points of~\eqref{eq:DCreg}.
\end{enumerate}
\end{thm}
See Appendix~\ref{sec:proof:nmPDCA_conv} for the proof of Theorem~\ref{theo:nmPDCA_conv}.
\begin{rem} \label{conv_rate}
We can ensure the optimal convergence rate of Algorithm~\ref{alg:APDCA} with the fixed step size for convex optimization~\eqref{eq:DCreg},
%under convexity.
though this is not the case with the $\ell_0$-constraint problem \eqref{eq:l2penalty}.
Since Algorithm~\ref{alg:APDCA} with the fixed step size is identical to Algorithm~\ref{alg:nmAPG}~\cite{li2015accelerated} if both $f$ and $g$ are convex,
the following convergence rate is guaranteed exactly the same as that of Algorithm~\ref{alg:nmAPG}.
% Therefore we %show 
%present the following theorem without proof.
%\begin{thm}\label{theo:APDCA_speed}
Let $\{\bvec x^{(t)}\}$ be the sequence generated by Algorithm~\ref{alg:APDCA} with the fixed step size $1/L$ and assume that $f$ and $g$ are convex. Then for any $\tau\ge 1$, we have
\begin{align*}
F(\bvec x^{(\tau+1)})-F(\bvec x^{\ast})\le \frac{2}{L(\tau+1)^2}\|\bvec x^{(0)}-\bvec x^{\ast}\|_2^2,
\end{align*}
where $\bvec x^{\ast}$ is a global minimizer of~\eqref{eq:DCreg}.
%\end{thm}
\end{rem}
%\begin{rem}
%Theorem \ref{theo:APDCA_speed} also holds for the nonmonotone APDCA with the same proof, which implies that the nonmonotone PDCA has the same convergence rate with the monotone APDCA if it is applied to convex programs.
%\end{rem}

%\subsection{Related works}
%\begin{itemize}
%\item \cite{lu2012sequential}: Sequential convex programming methods for a class of structured nonlinear programming
%\item \cite{yao2016efficient}:  Efficient learning with a family of nonconvex regularizers by redistributing nonconvexity
%\end{itemize}

\section{Numerical experiments}\label{sec:experiment}
In this section, we demonstrate the numerical performance of our algorithm.
All the computations were executed %using 
on a PC with 2.4GHz Intel CPU Core i7 and 16GB of memory.
\subsection{Comparison of two accelerations for PDCA on unconstrained sparse optimization}
We compared two types of acceleration for PDCA: APDCA (Algorithm~\ref{alg:APDCA}) and pDCA$_{e}$~\cite{wen2016proximal} (Algorithm~\ref{alg:pDCAe}).
For pDCA${}_e$, we used the program code of~\cite{wen2016proximal}, which is available at http:\slash\slash{}www.mypolyuweb.hk\slash\~{}tkpong\slash{}pDCAe\_final\_codes\slash.
In their code, $1/L=1/\lambda_{\max}(\bvec A^\top\bvec A)$ is employed as a step size and the extrapolation parameter $\{\beta_t\}$ is set to perform both the fixed and the adaptive restart strategy (for the details on how to choose the parameters, see Sections 3 and 5 in~\cite{wen2016proximal}).
Since their code is designed to solve the $\ell_{1-2}$ regularized linear regression problem:
\begin{align}
\min_{\bm{x}\in\mathbb R^n}F(\bvec x)=\frac{1}{2}\|\bvec A\bvec x-\bvec b\|_2^2+\rho(\|\bvec x\|_1-\|\bvec x\|_2),\label{eq:l1-l2}
\end{align}
where $\bvec A\in \mathbb R^{m\times n}$, $\bvec b\in\mathbb R^m$, and $\rho>0$, the comparison was %executed 
made on this problem.
We generated synthetic data following~~\cite{wen2016proximal}.
An $m\times n$ matrix $\bvec A$ was generated with i.i.d. standard Gaussian entries and then normalized so that each column $\bvec a_i$ of $\bvec A$ has unit norm, i.e., $\|\bvec a_i\|_2=1$.
Then a $k$-sparse vector $\bar{\bvec x}$ was generated to have i.i.d. standard Gaussian entries on an index subset of size $k$, which is chosen uniformly randomly from $\{1,\ldots,n\}$.
Finally, $\bvec b\in\mathbb R^m$ was generated by $\bvec b=\bvec A\bar{\bvec x}-0.01\cdot\bvec \varepsilon$, where $\bvec \varepsilon\in\mathbb R^m$ is a random vector with i.i.d. Gaussian entries.

We implemented two methods: APDCA${}_\mathrm{fix}$ and APDCA${}_\mathrm{bt}$.
APDCA${}_\mathrm{fix}$ represents Algorithm~\ref{alg:APDCA} with the fixed step size $1/L=1/\lambda_{\max}(\bvec A^\top\bvec A)$, while APDCA${}_\mathrm{bt}$ denotes Algorithm~\ref{alg:APDCA} with the backtracking line search with $\sigma=1.0\times 10^{-5}$.
We terminated these algorithms if the relative difference of the two successive objective values is less than $10^{-5}$.
%For pDCA${}_e$, we set $L=\lambda_{\max}(\bvec A^\top\bvec A)$ and the extrapolation parameter $\{\beta_t\}$ to perform both the fixed and the adaptive restart strategy (for the details on how to choose the parameters, see Section 3 and 5 in~\cite{wen2016proximal}).

The computational results on the synthetic data with $(m,n,k)=(720i,2560i,80i)$ for $i=1,\ldots,5$ are summarized in Table~\ref{table:APDCA-pDCAe},
where $t_L$ denotes the time for computing $L=\lambda_{\max}(\bvec A^\top \bvec A)$
\footnote{
The CPU times for APDCA${}_\mathrm{fix}$ and pDCA${}_e$ include $t_L$.
}
.
We can see that all the APDCAs require much fewer iterations than pDCA${}_e$, while pDCA${}_e$ achieves the best objective value.
pDCA${}_e$ converges in 1002 iterations for all the instances, which can be attributed to the fixed restart strategy employed at %with 
every 200 iterations %employed in 
of pDCA${}_e$.
The CPU time per iteration seems to depend on how many times the objective value is evaluated. %; actually, 
Actually, pDCA${}_e$ requires no evaluation of $F$, APDCA${}_\mathrm{fix}$ requires once, and APDCA${}_\mathrm{bt}$ requires %2 or 3 
a couple of times.
%Nevertheless, 
APDCA${}_\mathrm{fix}$ converges the fastest on the four out of five instances.

\begin{table}[h]
\caption{Results on~\eqref{eq:l1-l2} with synthetic data, $\lambda=5.0\times10^{-4}$, $\bvec x^{(0)}=\bvec 0$.}\label{table:APDCA-pDCAe}
\begin{center}
\begin{tabular}{|c|c|c|ccc|}\hline
size of $\bvec A$&$t_L$ (s)&method&objective value&time (s)&iteration\\\hline
&&APDCA${}_\mathrm{fix}$&3.18e-02&\textbf{1.8}&377\\\cline{3-6}
$2560\times 720$&0.3&APDCA${}_\mathrm{bt}$&3.06e-02&3.5&449\\\cline{3-6}
&&pDCA${}_e$&\textbf{3.05e-02}&2.9&1002\\\hline
&&APDCA${}_\mathrm{fix}$&\textbf{6.41e-02}&11.5&610\\\cline{3-6}
$5120\times 1440$&2.2&APDCA${}_\mathrm{bt}$&6.70e-02&\textbf{10.1}&331\\\cline{3-6}
&&pDCA${}_e$&\textbf{6.41e-02}&12.3&1002\\\hline
&&APDCA${}_\mathrm{fix}$&1.05e-01&\textbf{16.3}&409\\\cline{3-6}
$7680\times 2160$&2.4&APDCA${}_\mathrm{bt}$&\textbf{1.02e-01}&30.5&443\\\cline{3-6}
&&pDCA${}_e$&\textbf{1.02e-01}&24.2&1002\\\hline
&&APDCA${}_\mathrm{fix}$&\textbf{1.30e-01}&\textbf{42.0}&614\\\cline{3-6}
$10240\times 2880$&4.7&APDCA${}_\mathrm{bt}$&1.32e-01&44.9&372\\\cline{3-6}
&&pDCA${}_e$&\textbf{1.30e-01}&43.5&1002\\\hline
&&APDCA${}_\mathrm{fix}$&1.67e-01&\textbf{41.1}&363\\\cline{3-6}
$12800\times 3600$&8.1&APDCA${}_\mathrm{bt}$&1.63e-01&68.0&362\\\cline{3-6}
&&pDCA${}_e$&\textbf{1.60e-01}&68.4&1002\\\hline
\end{tabular}
\end{center}
\end{table}

\subsection{Results for constrained sparse optimization}
We compared the performance of our algorithms with various DCAs for the $\ell_0$-constrained optimization problem having some convex constraints.
In this section, we compare the following four methods for Examples~\ref{ex:sparsePCA}--\ref{ex:sparsennreg}.
\begin{itemize}
\item $\ell_2$-APDCA:
We apply Algorithm~\ref{alg:APDCA} with backtracking line search to the penalized problem~\eqref{eq:l2penalty}.
Since we cannot set $\rho$ in~\eqref{eq:l2penalty} as an exact penalty parameter, some errors tend to remain in the $(n-k)$ smallest components (in absolute value) of the output of APDCA.
Thus we round the output of APDCA to a $k$-sparse one by solving the small problem with $k$ variables obtained by fixing $(n-k)$ smallest components to $0$.

\item $\ell_2$-PDCA:
We apply Algorithm~\ref{alg:PDCAbt} to the penalized problem~\eqref{eq:l2penalty} and round the solution as in $\ell_2$-APDCA.

\item $\ell_1$-DCA~\cite{gotoh2015dc}:
%  We consider the top-$(k,1)$ penalized formulation~\eqref{eq:l1penalty}.
  We apply DC algorithm with the DC decomposition~\eqref{eq:naiveDCdecomp}
  to the top-$(k,1)$ penalized formulation~\eqref{eq:l1penalty}.
  %Then we solve it via DC algorithm with the DC decomposition~\eqref{eq:naiveDCdecomp}, where
 In the algorithm, the subproblems are solved by an optimization solver, IBM ILOG CPLEX 12.
\item MIP-DCA~\cite{thiao2010dc}:
We apply DCA to the transformed problem~\eqref{eq:concavepenalty}, where the big-M constant is fixed to 100.
The objective function in~\eqref{eq:concavepenalty} is decomposed as a DC function in the same way as~\eqref{eq:naiveDCdecomp}.
The DCA subproblems are solved by CPLEX.
\end{itemize}
Since a proper magnitude of penalty parameter $\rho$ depends on how we solve~\eqref{eq:l0const},
we tested $\rho=10^{i}$ $(i=0,\pm 1,\ldots,\pm 4)$ for each method and chose the one which attained the minimum objective value among $\rho$s that gave a $k$-sparse solution.
We again terminated all the algorithms if the relative difference of the two successive objective values is less than $10^{-5}$.

\subsubsection{Sparse principal component analysis}
We consider a sparse PCA
%principal component analysis
in Example~\ref{ex:sparsePCA}:
\begin{align*}
\min_{\bm{x}\in \mathbb R^n}\left\{-\bvec x^\top \bvec A\bvec x:\|\bvec x\|_0\le k,\ \|\bvec x\|_2\le 1\right\},
\end{align*}
where $\bvec A$ is an $n\times n$ positive semidefinite matrix.
We first %observe 
examined the dependency %to
on the initial solution $\bvec x^{(0)}$ %on
with the pit props data~\cite{jeffers1967two}, a standard benchmark to test the performance of algorithms for sparse PCA, whose correlation matrix has $n=13$. %calculated from 180 observations and has 13 explanatory variables.
%We apply 4 methods with the initial solution $\bvec x^{(0)}=\bvec 1/n$ and the target cardinality $k=5$.
%Figure~\ref{fig:PCApenalty} shows the objective value and the cardinality variance following $\rho$ in the range from $10^{-1}$ to $10^1$.
%We observe that $\ell_1$-DCA and MIP-DCA output more sparse solutions as $\rho$ becomes larger, while $\ell_2$-PDCA and $\ell_2$-APDCA always attain $k=5$ owing to the rounding step.
%\begin{figure}[h]\label{fig:PCApenalty}
%\begin{tabular}{cc}
%\begin{minipage}{0.5\hsize}
%\begin{center}
%\includegraphics[width=40mm]{PCAobj_rho.eps}
%\end{center}
%\begin{center}
%\end{center}\end{minipage}
%\begin{minipage}{0.5\hsize}
%\begin{center}
%\includegraphics[width=40mm]{PCAcard_rho.eps}
%\end{center}
%\end{minipage}
%\end{tabular}
%\caption{Objective value and cardinality following $\rho$ with $k=5$ for pit props data.}
%\end{figure}
%We observe the dependency to the initial solution $\bvec x^{(0)}$.
We randomly generated 100 initial points where $x_i^{(0)}\sim N(0,1)$ for $i=1,\ldots,n$.
Figure~\ref{fig:PCAinitial} shows the box plot of the objective values obtained by four algorithms with $k=5$, where $\rho=1$ was selected for all the algorithms.
We can see that $\ell_2$-PDCA and $\ell_2$-APDCA tend to achieve %smaller
better objective values and less dependency on the initial solution %, compared to 
than the other DCAs.
\begin{figure}[h]
\begin{center}
\includegraphics{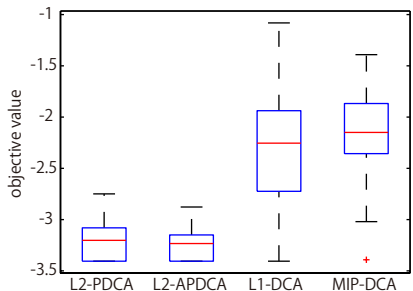}
\caption{Box plot of the objective values for 100 random initial solutions on the pit props data.}
\label{fig:PCAinitial}
\end{center}
\end{figure}

We show in Table~\ref{table:coloncancer} the results on the colon cancer data~\cite{alon1999broad}, which consists of 62 tissue samples with the gene expression profiles of $n=2000$ genes extracted from DNA micro-array data.
The parameters were fixed to $k=100$ and $\bvec x^{(0)}=\bvec 1/n$.
We can see that $\ell_2$-PDCA and $\ell_2$-APDCA converge faster owing to the light projection computations for the subproblems, while MIP-DCA achieves the best objective value.
\begin{table}[h]
  \caption{Results for sparse PCA with colon cancer data:
    the chosen $\rho$, the cardinality of the found solution,
    the attained %obtained 
    objective value, CPU time (sec.), and the number of iterations.
    %cardinality for colon cancer data.
  }\label{table:coloncancer}
\begin{center}
\begin{tabular}{|c|ccccc|}\hline
method&$\rho$&cardinality&objective value&time (s)&iteration\\\hline
$\ell_2$-PDCA&1000&100&-45.70&0.4&28\\\hline
$\ell_2$-APDCA&1000&100&-48.24&\textbf{0.2}&26\\\hline
$\ell_1$-DCA&1000&100&-45.65&3.3&3\\\hline
MIP-DCA&1000&100&\textbf{-80.68}&17.3&8\\\hline
\end{tabular}
\end{center}
\end{table}

\subsubsection{Sparse portfolio selection}
We consider a sparse portfolio selection problem in Example~\ref{ex:sparseportfolio}:
\begin{align*}
\min_{\bm{x}\in \mathbb R^n}\left\{\alpha\bvec x^\top\bvec V\bvec x-\bvec r^\top\bvec x:\|\bvec x\|_0\le k,\ \bvec 1^\top \bvec x=1\right\},
\end{align*}
where $\bvec V$ is a covariance matrix, $\bvec r$ is a mean return vector, and $\alpha$ $(>0)$ is a risk-aversion parameter.

We used the 2148 daily return vectors of 1338 stocks listed in the first section of Tokyo Stock Exchange (TSE) through February 2008 to November 2016.\footnote{This data set was collected through NEEDS-FinancialQUEST, a databank service provided by Nikkei Media Marketing, Inc., and was modified by deleting series of data which include missing values for the period.} % whose price data were not on the record.} 
We fixed the parameters as $\alpha=10$, $k=10$, and $\bvec x^{(0)}=\bvec 1/n$.
Table~\ref{table:tokyostock} reports the results on %Tokyo Stock Exchange (TSE)
the TSE %stock 
%price 
return data.
Our algorithms tend to require much more iterations but attain better objective values and converge much faster.
\begin{table}[h]
  \caption{Results for sparse portfolio selection with %Tokyo Stock Exchange
TSE %stock 
return 
%price 
data:
    the chosen $\rho$, the cardinality of the found solution,
    the %obtained 
    attained objective value, the obtained return, CPU time (sec.) and the number of iterations.
    %cardinality for colon cancer data.
    %$\rho$, cardinality, objective value, return $\bvec r^\top \bvec x$, time, iteration on Tokyo Stock Exchange data.
  }\label{table:tokyostock}
\begin{center}
{\small
\begin{tabular}{|c|cccccc|}\hline
method&$\rho$&cardinality&objective value&return&time (s)&iteration\\\hline
$\ell_2$-PDCA&1&10&-1.15e-04&8.05e-04&10.2&2523\\\hline
$\ell_2$-APDCA&1&10&\textbf{-3.91e-04}&\textbf{8.11e-04}&\textbf{2.3}&262\\\hline
$\ell_1$-DCA&1&9&7.75e-04&5.48e-04&45.8&2\\\hline
MIP-DCA&1&10&3.62e-04&6.03e-04&105.8&4\\\hline
\end{tabular}
}
\end{center}
\end{table}

\subsubsection{Sparse nonnegative least squares}
We consider a sparse nonnegative least squares problem in Example~\ref{ex:sparsennreg}:
\begin{align*}
\min_{\bm{x}\in \mathbb R^n}\left\{\frac{1}{2}\|\bvec A\bvec x-\bvec b\|_2^2:\|\bvec x\|_0\le k,\ x_i\ge 0\quad (i\in I)\right\},
\end{align*}
where $\bvec A\in \mathbb R^{m\times n}$, $\bvec b\in\mathbb R^m$,  and $I\subseteq\{1,\ldots,n\}$.

We report the results on synthetic data generated as follows.
Each column $\bm{a}_i$ of the matrix $\bm{A}^\top=(\bm{a}_1,\ldots,\bm{a}_m)$ was drawn independently from the normal distribution
 $N(\bm{0},\bm{\Sigma})$, where $\bm{\Sigma}=(\sigma_{ij})=(0.5^{|i-j|})$, and each column of $\bm{A}$ was then standardized, i.e, $\|\bm{a}_i\|_2=1$;
$\bm{b}$ was generated by $\bm{b}=\bm{A}\bar{\bm{x}}+\bm{\varepsilon}$, where $\bar{x}_i \sim U(-1,1)$ and $\varepsilon_i \sim N(0,1)$.

Table~\ref{table:nnreg} shows the results on synthetic data with various sizes.
We use $k=n/10$, $\bvec x^{(0)}=\bvec 1/n$, and $I=\{1,\ldots,\lfloor n/10\rfloor\}$.
We can see that $\ell_2$-PDCA is the fastest, and $\ell_2$-APDCA tends to find better solutions with smaller objective function values than the others. The additional steps %for accelerating
to accelerate PDCA also contribute to find better solutions. 
Based on the above observations, we may conclude that $\ell_2$-PDCA and $\ell_2$-APDCA find a good solution for constrained sparse optimization problems with % short
a small amount of computation time.
 %while $\ell_2$-APDCA requires more time and iterations to reach the best objective value.

\begin{table}[h]
  \caption{Results for sparse nonnegative least squares with synthetic data:
    the chosen $\rho$, the cardinality of the found solution,
    the obtained objective value, CPU time (sec.) and the number of iterations.
    %$\rho$, cardinality, objective value, time, iteration on synthetic data.
  }\label{table:nnreg}
\begin{center}
{\small
\begin{tabular}{|c|c|ccccc|}\hline
problem size&method&$\rho$&cardinality&objective value&time (s)&iteration\\\hline
&$\ell_2$-PDCA&1&20&1.48e-01&\textbf{0.1}&68\\\cline{2-7}
$640\times 180$&$\ell_2$-APDCA&1&20&\textbf{1.27e-01}&\textbf{0.1}&102\\\cline{2-7}
&$\ell_1$-DCA&0.1&20&1.45e-01&1.4&3\\\cline{2-7}
&MIP-DCA&100&20&4.48e-01&3.2&3\\\hline
&$\ell_2$-PDCA&1&40&1.36e-01&\textbf{0.2}&123\\\cline{2-7}
$1280\times 360$&$\ell_2$-APDCA&1&40&\textbf{1.10e-01}&0.3&125\\\cline{2-7}
&$\ell_1$-DCA&0.1&38&1.44e-01&7.8&4\\\cline{2-7}
&MIP-DCA&100&40&4.40e-01&17.0&3\\\hline
&$\ell_2$-PDCA&1&60&1.36e-01&\textbf{0.4}&97\\\cline{2-7}
$1920\times 540$&$\ell_2$-APDCA&1&60&\textbf{1.12e-01}&0.7&146\\\cline{2-7}
&$\ell_1$-DCA&0.1&54&1.41e-01&22.1&4\\\cline{2-7}
&MIP-DCA&100&60&2.20e-01&71.4&3\\\hline
&$\ell_2$-PDCA&1&80&1.39e-01&\textbf{0.9}&129\\\cline{2-7}
$2560\times 720$&$\ell_2$-APDCA&1&80&\textbf{1.04e-01}&1.3&156\\\cline{2-7}
&$\ell_1$-DCA&0.1&66&1.64e-01&43.2&3\\\cline{2-7}
&MIP-DCA&100&80&2.07e-02&123.6&2\\\hline
&$\ell_2$-PDCA&1&100&1.16e-01&\textbf{1.8}&173\\\cline{2-7}
$3200\times 900$&$\ell_2$-APDCA&1&100&\textbf{9.46e-02}&2.4&184\\\cline{2-7}
&$\ell_1$-DCA&0.1&80&1.50e-01&78.3&3\\\cline{2-7}
&MIP-DCA&100&100&4.41e-01&201.3&2\\\hline
\end{tabular}
}
\end{center}
\end{table}

\section{Conclusions}
In this paper, we have proposed an efficient DCA to solve $\ell_0$-constrained optimization problems having simple convex constraints. 
By introducing a new DC representation of the $\ell_0$-constraint, we have reduced the associated subproblem to the projection operation onto the convex constraint set, where the availability of closed-form solutions enables us to implement the operations very efficiently. 
% where subproblems are solved by
%closed-form solutions .
%In this paper, we have proposed an efficient DCA where subproblems are solved by
%closed-form solutions for $\ell_0$-constrained optimization problems having
%simple convex constraints.
%There are various existing methods for efficiently solving sparse optimization problems, but the efficiency is lost even when  convex constraints are added.
%By giving a new DC representation for the $\ell_0$-constraint, we have reduced its subproblems
%to the projection operation onto a convex set. 
%by the difference of convex (DC) functions, which 
Consequently, %The
the resulting DCA, called PDCA, still retains the efficiency even if
$\ell_0$-constrained optimization problems have some convex constraints.
Moreover, we have shown a link between PDCA and
Proximal Gradient Method (PGM), which leads to improvement of PDCA; the speed-up techniques
proposed for PGM such as the backtracking step size rule and 
the Nesterov's acceleration can be applied to PDCA. Indeed, the improved PDCA works
very well in numerical experiments, while retaining theoretical properties such as
the convergence to a stationary point of the input problem.

The techniques of PGMs have helped to speed up PDCA.
%The acceleration has improved the performance of PDCA, but
There are still a lot of issues that need to be addressed in the future.
Among those, some theoretical guarantee such as the convergence rate
discussed in Remark~\ref{conv_rate} 
 is the foremost one that needs to be investigated for nonconvex problem settings including
$\ell_0$-constrained optimization problems.
 Other speed-up techniques for PGMs such as adaptive restart strategy possibly improve
the performance of APDCA.
 
 \section*{Acknowledgements}
 We would like to thank Professor Ting Kei Pong for his comments on a manuscript and providing references.

\appendix
\section{Proofs of Propositions}
\subsection{Proof of Theorem \ref{theo:PDCA_conv}} 
\label{sec:proof:PDCA_conv}
To prove Theorem \ref{theo:PDCA_conv}, we provide the following lemma and proposition.
\begin{lemma}\label{lem:boundl}
In Algorithm~\ref{alg:PDCAbt}, $l^{(t)}$ is bounded for any $t\ge 0$.
\end{lemma}
\begin{proof}
From Assumption~\ref{assump_composite} (a), we have
\begin{align}\label{fbound}
f(\bvec x^{(t+1)})\le f(\bvec x^{(t)})+\langle \nabla f(\bvec x^{(t)}),\bvec x^{(t+1)}-\bvec x^{(t)} \rangle+\frac{L}{2}\|\bvec x^{(t+1)}-\bvec x^{(t)}\|_2^2.
\end{align}
Since $\bvec x^{(t+1)}$ is obtained by computing~\eqref{eq:PDCAbtsubprob}, we have
\begin{align}
g_1(\bvec x^{(t+1)})\le g_1(\bvec x^{(t)})-\langle \nabla f(\bvec x^{(t)})-\bvec s(\bvec x^{(t)}),\bvec x^{(t+1)}-\bvec x^{(t)} \rangle-\frac{l^{(t)}}{2}\|\bvec x^{(t+1)}-\bvec x^{(t)}\|_2^2.
\end{align}
It follows from the definition of the subgradient that
\begin{align}\label{g2bound}
g_2(\bvec x^{(t+1)})\ge g_2(\bvec x^{(t)})+\langle \bvec s(\bvec x^{(t)}),\bvec x^{(t+1)}-\bvec x^{(t)} \rangle.
\end{align}
Combining~\eqref{fbound}--\eqref{g2bound}, we have
\begin{align}\label{Fbound}
F(\bvec x^{(t+1)})\le F(\bvec x^{(t)})-\frac{l^{(t)}-L}{2}\|\bvec x^{(t+1)}-\bvec x^{(t)}\|_2^2.
\end{align}
Therefore, the criterion~\eqref{eq:backtracking} is satisfied when $l^{(t)}\ge L+\sigma$ and thus $l^{(t)}$ is bounded.
\end{proof}
\begin{proposition}[\cite{rockafellar2009variational}, Proposition 1 in the supplemental of~\cite{li2015accelerated}]\label{prop:semicontinuity}
Let $\{\bvec x^{(t)}\}$ and $\{\bvec u^{(t)}\}$ be sequences such that $\bvec x^{(t)}\to\bvec x^\ast$, $\bvec u^{(t)}\to\bvec u^\ast$, $g_1(\bvec x^{(t)})\to g_1(\bvec x^\ast)$, and $\bvec u^{(t)}\in\partial g_1(\bvec x^{(t)})$.
Then we have $\bvec u^\ast\in \partial g_1(\bvec x^\ast)$.
\end{proposition}

Now we are ready to prove Theorem~\ref{theo:PDCA_conv}.
It follows from~\eqref{eq:backtracking} that the sequence $\{F(\bvec x^{(t)})\}$ is %monotonically 
nonincreasing.
This, together with Assumption~\ref{assump_composite} (c), implies that $\lim_{t\to\infty}F(\bvec x^{(t)})$ exists.
Thus, % we have
%\begin{align*}
%\lim_{t\in \mathcal T\to\infty}F(\bvec x^{(t)})=F(\bvec x^{\ast}).
%\end{align*}
by taking limits on both sides of~\eqref{eq:backtracking}, we have
\begin{align}\label{limit}
\lim_{t\to \infty}\|\bvec x^{(t+1)}-\bvec x^{(t)}\|_2=0.
\end{align}
In addition, from Assumption~\ref{assump_composite} (c), the sequence $\{\bvec x^{(t)}\}$ is bounded.
Therefore, $\{\bvec x^{(t)}\}$ is a converging sequence, whose limit is denoted by $\bvec x^\ast$.

From the optimality condition of \eqref{eq:PDCAbtsubprob}, we have
\begin{align*}
\bvec 0\in \nabla f(\bvec x^{(t)})+l^{(t)}(\bvec x^{(t+1)}-\bvec x^{(t)})+\partial g_1(\bvec x^{(t+1)})-\bvec s(\bvec x^{(t)}),
\end{align*}
which is equivalent to
\begin{align}\label{eq:optcondx}
-\nabla f(\bvec x^{(t)})-l^{(t)}(\bvec x^{(t+1)}-\bvec x^{(t)})+\bvec s(\bvec x^{(t)})\in \partial g_1(\bvec x^{(t+1)}).
\end{align}
Since the sequence $\{\bvec s(\bvec x^{(t)})\}$ is bounded due to the continuity and convexity of $g_2$ and the boundedness of $\{\bvec x^{(t)}\}$,
there exists a subsequence $\mathcal T$ such that
$\bvec s^\ast :=\lim_{t\in\mathcal T\to\infty}\bvec s(\bvec x^{(t)})$ exists.
Note that $\bvec s^\ast\in\partial g_2(\bvec x^\ast)$ due to the closedness of $\partial g_2$.

Now we consider the sequence $\{-\nabla f(\bvec x^{(t)})-l^{(t)}(\bvec x^{(t+1)}-\bvec x^{(t)})+\bvec s(\bvec x^{(t)})\}$ of the left-hand side of~\eqref{eq:optcondx}.
From the continuity of $\nabla f$, the subsequential convergence of $\bvec s(\bvec x^{(t)})$ with respect to $\mathcal T$, \eqref{limit}, and Lemma~\ref{lem:boundl}, we have
\begin{align}\label{eq:limoptcond}
\lim_{t\in\mathcal T\to \infty}\left(-\nabla f(\bvec x^{(t)})-l^{(t)}(\bvec x^{(t+1)}-\bvec x^{(t)})+\bvec s(\bvec x^{(t)})\right)=-\nabla f(\bvec x^\ast)+\bvec s^\ast.
\end{align}

Then we will show that $g_1(\bvec x^{(t)})\to g_1(\bvec x^{\ast})$ to apply Proposition~\ref{prop:semicontinuity}.
We have from~\eqref{eq:PDCAbtsubprob} that
\begin{align}
\langle \nabla f(\bvec x^{(t)})-\bvec s(\bvec x^{(t)}),\bvec x^{(t+1)} \rangle+\frac{l^{(t)}}{2}\|\bvec x^{(t+1)}-\bvec x^{(t)}\|_2^2+g_1(\bvec x^{(t+1)})\nonumber\\
\le \langle \nabla f(\bvec x^{(t)})-\bvec s(\bvec x^{(t)}),\bvec x^{\ast} \rangle+\frac{l^{(t)}}{2}\|\bvec x^{\ast}-\bvec x^{(t)}\|_2^2+g_1(\bvec x^{\ast}).\label{eq:minimizerx}
\end{align}
Using \eqref{eq:minimizerx}, Lemma~\ref{lem:boundl}, the convergence of $\{\bvec x^{(t)}\}$, and the boundedness of $\{\nabla f(\bvec x^{(t)}) \}$ and $\{\bvec s(\bvec x^{(t)})\},$ we have
\begin{align*}
\limsup_{t\to\infty}g_1(\bvec x^{(t+1)})\le g_1(\bvec x^{\ast}).
\end{align*}
%On the other hand, the lower semicontinuousness of $g_1$ gives
Since $g_1$ is lower semicontinuous, i.e.,
\begin{align*}
\liminf_{t\to\infty}g_1(\bvec x^{(t+1)})\ge g_1(\bvec x^{\ast}),%.
\end{align*}
%Thus 
we have
\begin{align}
\lim_{t\to\infty}g_1(\bvec x^{(t+1)})= g_1(\bvec x^{\ast}).\label{eq:limg1x}
\end{align}
Finally, using \eqref{eq:limg1x}, \eqref{eq:limoptcond}, and Proposition~\ref{prop:semicontinuity} for \eqref{eq:optcondx}, we have
\begin{align*}
-\nabla f(\bvec x^\ast)+\bvec s^\ast\in \partial g_1(\bvec x^\ast),
\end{align*}
which implies
\begin{align}
\bvec 0&\in \{\nabla f(\bvec x^{\ast})\}+\partial g_1(\bvec x^{\ast})-\{\bvec s^{\ast}\}\subseteq \{\nabla f(\bvec x^{\ast})\}+\partial g_1(\bvec x^{\ast})-\partial g_2(\bvec x^{\ast}).\label{eq:stationary}
\end{align}
\hfill (End of Proof of Theorem \ref{theo:PDCA_conv})

\subsection{Proof of Theorem \ref{theo:nmPDCA_conv}} 
\label{sec:proof:nmPDCA_conv}

In the same manner to the proof of Theorem 4 in~\cite{li2015accelerated}, $\{\bvec x^{(t)}\}$ and $\{\bvec v^{(t)}\}$ can be proven to be bounded and
\begin{align}
\sum_{t_1\in\Omega_1}\|\bvec x^{(t_1+1)}-\bvec y^{(t_1)}\|_2^2+\sum_{t_2\in\Omega_2}\|\bvec v^{(t_2+1)}-\bvec x^{(t_2)}\|_2^2\le \frac{c^{(1)}-F^\ast}{\delta(1-\eta)}<\infty,\label{eq:thm4inNIPS}
\end{align}
where $F^\ast$ %is the minimum of 
denotes the optimal value of $F(\bvec x)$.
We consider the following three cases.

Case 1: $\Omega_2$ is finite.
In this case, there exists $T$ such that \eqref{eq:nmcriteria} is satisfied for all $t>T$.
Thus we have from~\eqref{eq:thm4inNIPS} that
\begin{align}
\sum_{t=T}^\infty\|\bvec x^{(t+1)}-\bvec y^{(t)}\|_2^2<\infty,\ \|\bvec x^{(t+1)}-\bvec y^{(t)}\|_2^2\to 0. \label{eq:xt+1-yt}
\end{align}
Since $\{\bvec x^{(t)}\}$ is bounded, we have that $\{\bvec y^{(t)}\}$ is bounded and thus has accumulation points, one of which is denoted by $\bvec y^\ast$, i.e., there exists a subsequence $\mathcal T$ such that
\begin{align*}
\lim_{t\in \mathcal T\to \infty}\bvec y^{(t)}=\bvec y^{\ast}.
\end{align*}
Then from~\eqref{eq:xt+1-yt}, we have $\bvec x^{(t+1)}\to \bvec y^\ast$ as $t\in \mathcal T\to \infty$.
From the optimality condition of~\eqref{eq:computez} and $\bvec x^{(t+1)}=\bvec z^{(t+1)}$,
\begin{align*}
\bvec 0\in \{\nabla f(\bvec y^{(t)})+l_y^{(t)}(\bvec x^{(t+1)}-\bvec y^{(t)})\}+\partial g_1(\bvec x^{(t+1)})-\{\bvec s(\bvec y^{(t)})\}.
\end{align*}
which is equivalent to
\begin{align}\label{eq:optcondv}
-\nabla f(\bvec y^{(t)})-l_y^{(t)}(\bvec x^{(t+1)}-\bvec y^{(t)})+\bvec s(\bvec y^{(t)})\in \partial g_1(\bvec x^{(t+1)}).
\end{align}
Now similarly to the proof of Theorem~\ref{theo:PDCA_conv}, we can take a subsequence $\mathcal T'$ of $\mathcal T$ such that
$\bvec s^\ast :=\lim_{t\in\mathcal T'\to\infty}\bvec s(\bvec y^{(t)})$ exists and $\bvec s^\ast\in\partial g_2(\bvec y^\ast)$.
Using this fact, in the same way as the derivation of~\eqref{eq:stationary}, we have
\begin{align*}
\bvec 0\in \{\nabla f(\bvec y^{\ast})\}+\partial g_1(\bvec y^{\ast})-\partial g_2(\bvec y^{\ast}).
\end{align*}
Since $\|\bvec x^{(t+1)}-\bvec y^{(t)}\|_2^2\to 0$, $\{\bvec x^{(t)}\}$ and $\{\bvec y^{(t)}\}$ have the same accumulation points and thus
\begin{align*}
\bvec 0\in \{\nabla f(\bvec x^{\ast})\}+\partial g_1(\bvec x^{\ast})-\partial g_2(\bvec x^{\ast}).
\end{align*}

Case 2: $\Omega_1$ is finite.
In this case, there exists $T$ such that \eqref{eq:nmcriteria} is not satisfied for all $t>T$.
Thus we have from~\eqref{eq:thm4inNIPS} that
\begin{align}
\sum_{t=T}^\infty\|\bvec v^{(t+1)}-\bvec x^{(t)}\|_2^2<\infty,\ \|\bvec v^{(t+1)}-\bvec x^{(t)}\|_2^2\to 0.\label{eq:vt+1-xt}
\end{align}
Then similarly to Case 1, for any accumulation point $\bvec x^\ast$ of $\{\bvec x^{(t)}\}$, we have
\begin{align*}
\bvec 0\in \{\nabla f(\bvec x^{\ast})\}+\partial g_1(\bvec x^{\ast})-\partial g_2(\bvec x^{\ast}).
\end{align*}

Case 3: $\Omega_1$ and $\Omega_2$ are both infinite.
In this case, we have
\begin{align*}
\|\bvec x^{(t_1+1)}-\bvec y^{(t_1)}\|_2^2\to 0,\ \|\bvec v^{(t_2+1)}-\bvec x^{(t_2)}\|_2^2\to 0,
\end{align*}
where $t_1\in\Omega_1$ and $t_2\in\Omega_2$.
Since $\{\bvec x^{(t)}\}$ is bounded, $\{\bvec y^{(t_1)}\}_{t_1\in\Omega_1}$ is also bounded.
Now similarly to Cases 1 and 2, any accumulation point $\bvec y^\ast$ of $\{\bvec y^{(t_1)}\}_{t_1\in\Omega_1}$ and any accumulation point $\bvec x^\ast$ of $\{\bvec x^{(n_i)}\}_{n_i\in\Omega_2}$ are stationary points of \eqref{eq:DCreg}.
In addition, $\{\bvec x^{(t_1+1)}\}_{t_1\in\Omega_1}$ and $\{\bvec y^{(t_1)}\}_{t_1\in\Omega_1}$ have the same accumulation point and thus any accumulation point $\bvec x^\ast$ of $\{\bvec x^{(t_1+1)}\}_{t_1\in\Omega_1}$ is also a stationary points of \eqref{eq:DCreg}.
Similarly, any accumulation point $\bvec v^\ast$ of $\{\bvec v^{(t_2+1)}\}_{t_2\in\Omega_2}$ is a stationary points of \eqref{eq:DCreg}.

\hfill (End of Proof of Theorem \ref{theo:nmPDCA_conv})

%%%%%%%%%%%%%%%
\bibliography{const_sparse}
\end{document}